\documentclass[12pt,reqno]{smfart}
\usepackage{amsmath,amssymb,smfthm,stmaryrd,epic,enumerate}
\usepackage[frenchb]{babel}
\usepackage[latin1]{inputenc}
\usepackage[all]{xy}
\usepackage{a4wide}

\usepackage{mathrsfs}
\usepackage[active]{srcltx}

\usepackage{lmodern}

\usepackage{hyperref}

\NumberTheoremsIn{section}\SwapTheoremNumbers

\begin{document}
\newtheorem{prop-defi}[smfthm]{Proposition-DÈfinition}
\newtheorem{notas}[smfthm]{Notations}
\newtheorem{nota}[smfthm]{Notation}
\newtheorem{defis}[smfthm]{DÈfinitions}
\newtheorem{hypo}[smfthm]{HypothËse}
\newtheorem*{theo*}{ThÈorËme}
\newtheorem*{hyp*}{HypothËses}

\def\Tm{{\mathbb T}}
\def\Um{{\mathbb U}}
\def\Am{{\mathbb A}}
\def\Fm{{\mathbb F}}
\def\Mm{{\mathbb M}}
\def\Nm{{\mathbb N}}
\def\Pm{{\mathbb P}}
\def\Qm{{\mathbb Q}}
\def\Zm{{\mathbb Z}}
\def\Dm{{\mathbb D}}
\def\Cm{{\mathbb C}}
\def\Rm{{\mathbb R}}
\def\Gm{{\mathbb G}}
\def\Lm{{\mathbb L}}
\def\Km{{\mathbb K}}
\def\Om{{\mathbb O}}
\def\Em{{\mathbb E}}
\def\Xm{{\mathbb X}}

\def\BC{{\mathcal B}}
\def\QC{{\mathcal Q}}
\def\TC{{\mathcal T}}
\def\ZC{{\mathcal Z}}
\def\AC{{\mathcal A}}
\def\CC{{\mathcal C}}
\def\DC{{\mathcal D}}
\def\EC{{\mathcal E}}
\def\FC{{\mathcal F}}
\def\GC{{\mathcal G}}
\def\HC{{\mathcal H}}
\def\IC{{\mathcal I}}
\def\JC{{\mathcal J}}
\def\KC{{\mathcal K}}
\def\LC{{\mathcal L}}
\def\MC{{\mathcal M}}
\def\NC{{\mathcal N}}
\def\OC{{\mathcal O}}
\def\PC{{\mathcal P}}
\def\UC{{\mathcal U}}
\def\VC{{\mathcal V}}
\def\XC{{\mathcal X}}
\def\SC{{\mathcal S}}

\def\BF{{\mathfrak B}}
\def\AF{{\mathfrak A}}
\def\GF{{\mathfrak G}}
\def\EF{{\mathfrak E}}
\def\CF{{\mathfrak C}}
\def\DF{{\mathfrak D}}
\def\JF{{\mathfrak J}}
\def\LF{{\mathfrak L}}
\def\MF{{\mathfrak M}}
\def\NF{{\mathfrak N}}
\def\XF{{\mathfrak X}}
\def\UF{{\mathfrak U}}
\def\KF{{\mathfrak K}}
\def\FF{{\mathfrak F}}

\def \longmapright#1{\smash{\mathop{\longrightarrow}\limits^{#1}}}
\def \mapright#1{\smash{\mathop{\rightarrow}\limits^{#1}}}
\def \lexp#1#2{\kern \scriptspace \vphantom{#2}^{#1}\kern-\scriptspace#2}
\def \linf#1#2{\kern \scriptspace \vphantom{#2}_{#1}\kern-\scriptspace#2}
\def \linexp#1#2#3 {\kern \scriptspace{#3}_{#1}^{#2} \kern-\scriptspace #3}

\def \Ext{\mathop{\mathrm{Ext}}\nolimits}
\def \ad{\mathop{\mathrm{ad}}\nolimits}
\def \sh{\mathop{\mathrm{Sh}}\nolimits}
\def \irr{\mathop{\mathrm{Irr}}\nolimits}
\def \FH{\mathop{\mathrm{FH}}\nolimits}
\def \FPH{\mathop{\mathrm{FPH}}\nolimits}
\def \coh{\mathop{\mathrm{Coh}}\nolimits}
\def \res{\mathop{\mathrm{Res}}\nolimits}
\def \op{\mathop{\mathrm{op}}\nolimits}
\def \rec {\mathop{\mathrm{rec}}\nolimits}
\def \art{\mathop{\mathrm{Art}}\nolimits}
\def \vol {\mathop{\mathrm{vol}}\nolimits}
\def \cusp {\mathop{\mathrm{Cusp}}\nolimits}
\def \scusp {\mathop{\mathrm{Scusp}}\nolimits}
\def \Iw {\mathop{\mathrm{Iw}}\nolimits}
\def \JL {\mathop{\mathrm{JL}}\nolimits}
\def \speh {\mathop{\mathrm{Speh}}\nolimits}
\def \isom {\mathop{\mathrm{Isom}}\nolimits}
\def \Vect {\mathop{\mathrm{Vect}}\nolimits}
\def \groth {\mathop{\mathrm{Groth}}\nolimits}
\def \hom {\mathop{\mathrm{Hom}}\nolimits}
\def \deg {\mathop{\mathrm{deg}}\nolimits}
\def \val {\mathop{\mathrm{val}}\nolimits}
\def \det {\mathop{\mathrm{det}}\nolimits}
\def \rep {\mathop{\mathrm{Rep}}\nolimits}
\def \spec {\mathop{\mathrm{Spec}}\nolimits}
\def \fr {\mathop{\mathrm{Fr}}\nolimits}
\def \frob {\mathop{\mathrm{Frob}}\nolimits}
\def \ker {\mathop{\mathrm{Ker}}\nolimits}
\def \im {\mathop{\mathrm{Im}}\nolimits}
\def \Red {\mathop{\mathrm{Red}}\nolimits}
\def \red {\mathop{\mathrm{red}}\nolimits}
\def \aut {\mathop{\mathrm{Aut}}\nolimits}
\def \diag {\mathop{\mathrm{diag}}\nolimits}
\def \spf {\mathop{\mathrm{Spf}}\nolimits}
\def \Def {\mathop{\mathrm{Def}}\nolimits}
\def \twist {\mathop{\mathrm{Twist}}\nolimits}
\def \supp {\mathop{\mathrm{Supp}}\nolimits}
\def \Id {{\mathop{\mathrm{Id}}\nolimits}}
\def \lie {{\mathop{\mathrm{Lie}}\nolimits}}
\def \Ind{\mathop{\mathrm{Ind}}\nolimits}
\def \ind {\mathop{\mathrm{ind}}\nolimits}
\def \bad {\mathop{\mathrm{Bad}}\nolimits}
\def \top {\mathop{\mathrm{Top}}\nolimits}
\def \ker {\mathop{\mathrm{Ker}}\nolimits}
\def \coker {\mathop{\mathrm{Coker}}\nolimits}
\def \gal {{\mathop{\mathrm{Gal}}\nolimits}}
\def \Nr {{\mathop{\mathrm{Nr}}\nolimits}}
\def \rn {{\mathop{\mathrm{rn}}\nolimits}}
\def \tr {{\mathop{\mathrm{Tr~}}\nolimits}}
\def \Sp {{\mathop{\mathrm{Sp}}\nolimits}}
\def \st {{\mathop{\mathrm{St}}\nolimits}}
\def \sp{{\mathop{\mathrm{Sp}}\nolimits}}
\def \perv{\mathop{\mathrm{Perv}}\nolimits}
\def \tor {{\mathop{\mathrm{Tor}}\nolimits}}
\def \gr {{\mathop{\mathrm{gr}}\nolimits}}
\def \nilp {{\mathop{\mathrm{Nilp}}\nolimits}}
\def \obj {{\mathop{\mathrm{Obj}}\nolimits}}
\def \spl {{\mathop{\mathrm{Spl}}\nolimits}}

\def \rem{{\noindent\textit{Remarque.~}}}
\def \rems{{\noindent\textit{Remarques:~}}}
\def \ext {{\mathop{\mathrm{Ext}}\nolimits}}
\def \End {{\mathop{\mathrm{End}}\nolimits}}

\def\semi{\mathrel{>\!\!\!\triangleleft}}
\let \DS=\displaystyle
\def\HT{{\mathop{\mathcal{HT}}\nolimits}}

\def \hi{\HC}
\newcommand*{\tarrow}{\relbar\joinrel\mid\joinrel\twoheadrightarrow}
\newcommand*{\harrow}{\lhook\joinrel\relbar\joinrel\mid\joinrel\rightarrow}
\newcommand*{\rarrow}{\relbar\joinrel\mid\joinrel\rightarrow}
\def \coim {{\mathop{\mathrm{Coim}}\nolimits}}
\def \can {{\mathop{\mathrm{can}}\nolimits}}
\def\LFF{{\mathscr L}}

\setcounter{secnumdepth}{3} \setcounter{tocdepth}{3}

\def \Fil{\mathop{\mathrm{Fil}}\nolimits}
\def \CoFil{\mathop{\mathrm{CoFil}}\nolimits}
\def \Fill{\mathop{\mathrm{Fill}}\nolimits}
\def \CoFill{\mathop{\mathrm{CoFill}}\nolimits}
\def\SF{{\mathfrak S}}
\def\PF{{\mathfrak P}}
\def \EFil{\mathop{\mathrm{EFil}}\nolimits}
\def \EFill{\mathop{\mathrm{EFill}}\nolimits}
\def \FP{\mathop{\mathrm{FP}}\nolimits}

\let \longto=\longrightarrow
\let \oo=\infty

\let \d=\delta
\let \k=\kappa

\newcommand{\marque}{\addtocounter{smfthm}{1}
{\smallskip \noindent \textit{\thesmfthm}~---~}}

\renewcommand\atop[2]{\ensuremath{\genfrac..{0pt}{1}{#1}{#2}}}

\title[Mirabolique stratification de Newton]{Groupe mirabolique, stratification de
Newton raffinÈe et cohomologie des espaces de Lubin-Tate}

\alttitle{Mirabolic group, ramified Newton stratification and cohomology of Lubin-Tate spaces}

\author{Boyer Pascal}
\email{boyer@math.univ-paris13.fr}
\address{UniversitÈ Paris 13, Sorbonne Paris CitÈ \\
LAGA, CNRS, UMR 7539\\ 
F-93430, Villetaneuse (France) \\
PerCoLaTor: ANR-14-CE25}

\thanks{L'auteur remercie l'ANR pour son soutien dans le cadre du projet PerCoLaTor 14-CE25.}

\frontmatter

\begin{abstract}
Dans \cite{boyer-invent2}, on dÈtermine les groupes de cohomologie des espaces de Lubin-Tate 
par voie globale en calculant les fibres des faisceaux de cohomologie du faisceau pervers
des cycles Èvanescents $\Psi$ d'une variÈtÈ de Shimura de type Kottwitz-Harris-Taylor. L'ingrÈdient
le plus complexe consiste ‡ contrÙler les flËches de deux suites spectrales calculant l'une les faisceaux
de cohomologie des faisceaux pervers d'Harris-Taylor, et l'autre ceux de $\Psi$. Dans cet article, 
nous contournons ces difficultÈs en utilisant la thÈorie classique des reprÈsentations du groupe mirabolique ainsi qu'un 
argument gÈomÈtrique simple.

\end{abstract}

\begin{altabstract}
In \cite{boyer-invent2}, we determine the cohomology of Lubin-Tate spaces globally using the comparison
theorem of Berkovich by computing the fibers at supersingular points of the perverse sheaf of vanishing
cycle $\Psi$ of some Shimura variety of Kottwitz-Harris-Taylor type. The most difficult argument deals
with the control of maps of the spectral sequences computing the sheaf cohomology of both
Harris-Taylor perverse sheaves and those of $\Psi$. In this paper, we bypass these difficulties using
the classical theory of representations of the mirabolic group and a simple geometric argument.

\end{altabstract}



\altkeywords{Shimura varieties, torsion in the cohomology, maximal ideal of the Hecke algebra,
localized cohomology, galois representation}

\maketitle

\pagestyle{headings} \pagenumbering{arabic}

\tableofcontents
%
%
\renewcommand{\theequation}{\arabic{section}.\arabic{smfthm}}

\section{Introduction}
\renewcommand{\thesmfthm}{\arabic{section}.\arabic{smfthm}}

Le rÈsultat principal de \cite{boyer-invent2} est la dÈtermination de chacun des groupes de cohomologie
‡ coefficients dans $\overline \Qm_l$, de la tour de Lubin-Tate. En utilisant le thÈorËme de comparaison
de Berkovich, la dÈmonstration est de nature globale et consiste ‡ calculer les germes en un point supersingulier
des faisceaux de cohomologie du complexe des cycles Èvanescents $\Psi_\IC$
en une place $v$ d'un corps CM, $F$, d'une tour de variÈtÈs de Shimura de type Kottwitz-Harris-Taylor 
$X_\IC$ indexÈe par l'ensemble $\IC$ de ses niveaux. 
La preuve se dÈroule alors en deux temps relativement distincts:
\begin{itemize}
\item Les auteurs de \cite{h-t} associent ‡
une reprÈsentation irrÈductible $\rho_v$ du groupe des inversibles $D_{v,h}^\times$
de l'algËbre ‡ division centrale sur $F_v$ d'invariant $1/h$, un systËme local $\LC(\rho_v)$, 
dit d'Harris-Taylor, sur la strate de Newton $X^{=h}_{\IC,\bar s_v}$ de la fibre spÈciale en $v$ de $X_\IC$
et calculent la somme alternÈe de leurs groupes de cohomologies ‡ support compact.
Dans \cite{boyer-invent2} on exploite ce calcul en exprimant les images des faisceaux pervers 
$\lexp p j^{=h}_! \LC(\rho_v)[d-h]$ et $\Psi_{\IC}$ dans un groupe de Grothendieck
de faisceaux pervers Èquivariants, en termes d'extensions intermÈdiaires de systËmes locaux
d'Harris-Taylor sur les diffÈrentes strates de Newton. 
Il s'agit de la partie la plus simple de loc. cit., celle qui est contrÙlÈe par la formule des traces.

\item Dans un deuxiËme temps, on Ètudie les suites spectrales associÈes ‡ la filtration par les poids de
ces deux faisceaux pervers pour calculer les faisceaux de cohomologie de 
$\lexp p j^{=h}_{!*} \LC(\rho_v)[d-h]$ et ceux de $\Psi_{\IC}$.
La partie la plus complexe de \cite{boyer-invent2} consiste ‡ contrÙler les flËches de ces suites spectrales
et au final, montrer \og qu'elles sont le moins triviales possibles \fg, au sens o˘ dËs que la source et le
but d'une telle flËche partagent un sous-faisceau Èquivariant,
cette flËche induit un isomorphisme sur ceux-ci. 
\end{itemize}
Pour montrer ce fait, dans \cite{boyer-invent2},
on utilise une propriÈtÈ de compatibilitÈ ‡ l'involution de Zelevinsky de la cohomologie des espaces
de Lubin-Tate. Dans \cite{boyer-compositio}, on propose une autre dÈmonstration
plus simple conceptuellement mais aussi trËs lourde, reposant sur des calculs de groupes de cohomologie 
globaux et en utilisant le thÈorËme de Lefschetz vache. 

Ces deux preuves sont techniquement difficiles et ne permettent pas de rÈellement comprendre 
la raison profonde de la non trivialitÈ de ces flËches. Le but premier de cet article
est ainsi de proposer une nouvelle preuve simple et en un sens, naturelle.
Plus prÈcisÈment ‡ partir de la premiËre Ètape mentionnÈe ci-avant, nous utilisons tout d'abord, 
avec les notations du paragraphe suivant, le fait que l'inclusion 
$$X^{\geq h}_{\IC,\bar s_v,\overline{1_h}} \setminus
X^{\geq h+1}_{\IC,\bar s_v,\overline{1_{h+1}}} \hookrightarrow X^{\geq h}_{\IC,\bar s_v,\overline{1_h}}$$
est affine, cf. le lemme \ref{lem-jaffine}. Ensuite de la thÈorie classique des reprÈsentations induites du 
groupe mirabolique, rappelÈe au \S \ref{para-mira}, on obtient, cf. les suites exactes courtes (\ref{eq-sec0}) 
et (\ref{eq-sec1}), une filtration, ‡ deux crans, des faisceaux pervers d'Harris-Taylor, fournissant une suite 
spectrale qui dÈgÈnËre en $E_1$ et qui correspond au terme $E_2=E_\oo$ de la suite spectrale 
de \cite{boyer-invent2}. Autrement dit la dÈtermination
des flËches dans \cite{boyer-invent2} est entiËrement contrÙlÈe par l'affinitÈ de l'inclusion ci-dessus
et la thÈorie des reprÈsentations du groupe mirabolique.
Le mÍme procÈdÈ appliquÈ au faisceau pervers des cycles Èvanescents fournit de la mÍme faÁon, 
cf. la proposition \ref{prop-psi-fil}, une filtration de celui-ci, dont la suite spectrale calculant ses faisceaux 
de cohomologie dÈgÈnËre en $E_1$, cf. le corollaire \ref{coro-fin}, et coÔncide avec le terme $E_2=E_\oo$ 
de la suite spectrale de \cite{boyer-invent2}.

Je remercie profondÈment J.-F. Dat pour l'intÈrÍt qu'il porte ‡ mon travail, ses questions et ses
suggestions qui ont sÈrieusement contribuÈ ‡ l'amÈlioration du texte.

\section{GÈomÈtrie des variÈtÈs de Shimura unitaires simples}

Soit $F=F^+ E$ un corps CM avec $E/\Qm$ quadratique imaginaire, dont on fixe 
un plongement rÈel $\tau:F^+ \hookrightarrow \Rm$. Pour $v$ une place de $F$, on notera 
\begin{itemize}
\item $F_v$ le complÈtÈ du localisÈ de $F$ en $v$,

\item $\OC_v$ l'anneau des entiers de $F_v$,

\item $\varpi_v$ une uniformisante et

\item $q_v$ le cardinal du corps rÈsiduel $\kappa(v)=\OC_v/(\varpi_v)$.
\end{itemize}

Soit $B$ une algËbre ‡ 
division centrale sur $F$ de dimension $d^2$ telle qu'en toute place $x$ de $F$,
$B_x$ est soit dÈcomposÈe soit une algËbre ‡ division et on suppose $B$ 
munie d'une involution de
seconde espËce $*$ telle que $*_{|F}$ est la conjugaison complexe $c$. Pour
$\beta \in B^{*=-1}$, on note $\sharp_\beta$ l'involution $x \mapsto x^{\sharp_\beta}=\beta x^*
\beta^{-1}$ et $G/\Qm$ le groupe de similitudes dÈfini pour toute $\Qm$-algËbre $R$ par 
$$
G(R)  \simeq   \{ (\lambda,g) \in R^\times \times (B^{op} \otimes_\Qm R)^\times  \hbox{ tel que } 
gg^{\sharp_\beta}=\lambda \}
$$
avec $B^{op}=B \otimes_{F,c} F$. 
Si $x$ est une place de $\Qm$ dÈcomposÈe $x=yy^c$ dans $E$ alors 
\addtocounter{smfthm}{1}
\begin{equation} \label{eq-facteur-v}
G(\Qm_x) \simeq (B_y^{op})^\times \times \Qm_x^\times \simeq \Qm_x^\times \times
\prod_{z_i} (B_{z_i}^{op})^\times,
\end{equation}
o˘, en identifiant les places de $F^+$ au dessus de $x$ avec les places de $F$ au dessus de $y$,
$x=\prod_i z_i$ dans $F^+$.

Dans \cite{h-t}, les auteurs justifient l'existence d'un $G$ comme ci-dessus tel qu'en outre:
\begin{itemize}
\item si $x$ est une place de $\Qm$ qui n'est pas dÈcomposÈe dans $E$ alors
$G(\Qm_x)$ est quasi-dÈployÈ;

\item les invariants de $G(\Rm)$ sont $(1,d-1)$ pour le plongement $\tau$ et $(0,d)$ pour les
autres. 
\end{itemize}

\begin{nota} 
On fixe un nombre premier $l$ non ramifiÈ dans $E$ et
on note $\spl$ l'ensemble des places $v$ de $F$ telles que $p_v:=v_{|\Qm} \neq l$
est dÈcomposÈ dans $F$ et $B_v^\times \simeq GL_d(F_v)$.
\end{nota}

Rappelons, cf. \cite{h-t} bas de la page 90, qu'un sous-groupe ouvert compact de $G(\Am^\oo)$
est dit \og assez petit \fg{} s'il existe une place $x$ pour laquelle la projection de $U^v$ 
sur $G(\Qm_x)$ ne contienne aucun ÈlÈment d'ordre fini autre que l'identitÈ.

\begin{nota}
Soit $\IC$ l'ensemble des sous-groupes compacts ouverts \og assez petits \fg{} de $G(\Am^\oo)$.
Pour $I \in \IC$, on note $X_{I,\eta} \longrightarrow \spec F$ la variÈtÈ de Shimura associÈe, dit
de Kottwitz-Harris-Taylor.
\end{nota}

\rem Pour tout $v \in \spl$, la variÈtÈ $X_{I,\eta}$ admet un modËle projectif $X_{I,v}$ sur $\spec \OC_v$
de fibre spÈciale $X_{I,s_v}$. Pour $I$ dÈcrivant $\IC$, le systËme projectif $(X_{I,v})_{I\in \IC}$ 
est naturellement muni d'une action de $G(\Am^\oo) \times \Zm$  telle que l'action d'un ÈlÈment
$w_v$ du groupe de Weil $W_v$ de $F_v$ est donnÈe par celle de $-\deg (w_v) \in \Zm$,
o˘ $\deg=\val \circ \art^{-1}$ o˘ $\art^{-1}:W_v^{ab} \simeq F_v^\times$ est
l'isomorphisme d'Artin qui envoie les Frobenius gÈomÈtriques sur les uniformisantes.

\begin{notas} (cf. \cite{boyer-invent2} \S 1.3) \label{nota-jh0}
Pour $I \in \IC$, la fibre spÈciale gÈomÈtrique $X_{I,\bar s_v}$ admet une stratification de Newton
$$X_{I,\bar s_v}=:X^{\geq 1}_{I,\bar s_v} \supset X^{\geq 2}_{I,\bar s_v} \supset \cdots \supset 
X^{\geq d}_{I,\bar s_v}$$
o˘ $X^{=h}_{I,\bar s_v}:=X^{\geq h}_{I,\bar s_v} - X^{\geq h+1}_{I,\bar s_v}$ est un schÈma 
affine, lisse de pure dimension $d-h$ formÈ des points gÈomÈtriques dont la partie connexe du groupe de 
Barsotti-Tate est de rang $h$.
Pour tout $1 \leq h <d$, nous utiliserons les notations suivantes:
$$i^{h}:X^{\geq h}_{I,\bar s_v} \hookrightarrow X_{I,\bar s_v}, \quad
j^{\geq h}: X^{=h}_{I,\bar s_v} \hookrightarrow X^{\geq h}_{I,\bar s_v},$$
ainsi que $j^{=h}=i^h \circ j^{\geq h}$.
\end{notas}

Soit $\GC(h)$ le groupe de Barsotti-Tate universel sur $X_{I,\bar s,\overline{1_h}}^{=h}$:
$$0 \rightarrow \GC(h)^c \longrightarrow \GC(h) \longrightarrow \GC(h)^{et} \rightarrow 0$$
o˘ $\GC(h)^c$ (resp. $\GC(h)^{et}$) est connexe (resp. Ètale) de dimension $h$ (resp. $d-h$).
Notons 
$\iota_{m_1}:(\PC_v^{-m_1}/\OC_v)^d \longrightarrow \GC(h)[p^{m_1}]$
la structure de niveau universelle. Notons alors $(e_i(m_1))_{1 \leq i \leq d}$ la base canonique de
$(\PC_v^{-m_1}/\OC_v)^d$.

\begin{defi} \label{defi-strate1}
On introduit le sous-schÈma fermÈ 
$X_{I,\bar s,\overline{1_h}}^{=h}$ de $X_{I,\bar s_v}^{=h}$ dÈfini par la propriÈtÈ
que $\bigl \{ \iota_{m_1}(e_i(m_1)):~1 \leq i \leq h \bigr \}$ forme une base de Drinfeld de 
$\GC(h)^c[p^{m_1}]$.

Plus gÈnÈralement Ètant donnÈ $a \in GL_d(\OC_v)/P_{h,d-h}(\OC_v)$ que l'on identifiera avec
le sous-espace vectoriel $\langle a(e_1),\cdots,a(e_h) \rangle$ engendrÈ par les images par $a$ des
$h$ premiers vecteurs $e_1,\cdots, e_h$ de la base canonique de $\OC_v^d$, on dÈfinit la strate 
\emph{dite pure}\footnote{en comparaison des strates de la dÈfinition \ref{defi-strate2}} 
$X_{I,\bar s,a}^{=h}$ comme le sous-schÈma fermÈ tel que
$\bigl \{ \iota_{m_1}(a(e_i(m_1)):~1 \leq i \leq h \}$ forme une base de Drinfeld de 
$\GC(h)^c[p^{m_1}]$, o˘ $e_i(m_1)$ est l'image de $e_i$ modulo $\PC_v^{m_1}$.
\end{defi}

\rem L'action de $GL_d(F_v)$ sur $X^{=h}_{I,\bar s_v}$ donne la dÈcomposition
\addtocounter{smfthm}{1}
\begin{equation} \label{eq-strate-induite}
X_{I,\bar s_v}^{=h} \simeq X_{I,\bar s_v,\overline{1_h}}^{=h} 
\times_{P_{h,d}(\OC_v/(\varpi_v^n))} GL_d(\OC_v/(\varpi_v^n)),
\end{equation}
au sens o˘ la strate $X_{I,\bar s_v,a}^{=h}$ s'obtient comme l'image
par $a$ de $X_{I,\bar s,\overline{1_h}}^{=h}$.

%
%

\begin{nota} \label{nota-jh1}
On note $X_{I,\bar s_v,\overline{1_h}}^{\geq h}$ l'adhÈrence de 
$X_{I,\bar s_v,\overline{1_h}}^{=h}$ dans $X_{I,\bar s_v}^{\geq h}$ et 
$$j^{\geq h}_{\overline{1_h}}: X_{I,\bar s_v,\overline{1_h}}^{=h} \hookrightarrow 
X_{I,\bar s_v,\overline{1_h}}^{\geq h},$$
ainsi que $j^{= h}_{\overline{1_h}}:=i^h_{\overline{1_h}} \circ j^{\geq h}_{\overline{1_h}}$
o˘ $i^h_{\overline{1_h}}:X_{I,\bar s_v,\overline{1_h}}^{\geq h} \hookrightarrow X_{I,\bar s_v}$.
\end{nota}

\section{Faisceaux pervers d'Harris-Taylor et des cycles Èvanescents}

Fixons une place $v \in \spl$. En utilisant un analogue des classiques variÈtÈs d'Igusa,
les auteurs de \cite{h-t} associent ‡ toute reprÈsentation $\rho_v$ de l'ordre maximal $\DC_{v,h}^\times$
de $D_{v,h}^\times$, un systËme local $\LC(\rho_v)_{\overline{1_h}}$ 
sur $X^{=h}_{\IC,\bar s_v,\overline{1_h}}$ 
muni d'une action Èquivariante de $G(\Am^{\oo,p}) \times P_{h,d}(F_v) \times \Zm$ telle que 
\begin{itemize}
\item le sous-groupe unipotent de $P_{h,d}(F_v)$ agit trivialement et

\item l'action du facteur $GL_h(F_v)$ de son Levi agit via 
$\val\circ \det: GL_h(F_v) \twoheadrightarrow \Zm$. 
\end{itemize}
On introduit alors la version induite de ces systËmes locaux
$$\LC(\rho_v):= \LC(\rho_v)_{\overline{1_h}} \times_{P_{h,d}(F_v)} GL_d(F_v),$$
muni donc d'une action Èquivariante de $G(\Am^{\oo,p}) \times GL_d(F_v) \times \Zm$.

\begin{nota}
La correspondance de Jacquet-Langlands permet de paramÈtrer les $\overline \Qm_l$-reprÈsentations
irrÈductibles de $D_{v,h}^\times$ ‡ l'aide des reprÈsentations irrÈductibles cuspidales
$\pi_v$ de $GL_g(F_v)$ pour $g$ un diviseur de $h=tg$, on Ècrit une telle reprÈsentation sous la forme
$\pi_v[t]_D$.
\end{nota}

\begin{nota} \label{nota-HT} 
Pour $\pi_v$ une reprÈsentation irrÈductible cuspidale de $GL_g(F_v) $et $\Pi_t$ une reprÈsentation de 
$GL_{tg}(F_v)$, on note\footnote{Lorsque $n=0$, on le fera disparaitre des notations.}, cf. la premiËre remarque
de 2.1.3 de \cite{boyer-invent2}, 
$$\widetilde{HT_{\overline{1_{tg}}}}(\pi_v,\Pi_t)(n):=\LC(\pi_v[t]_D)_{\overline{1_{tg}}} 
\otimes \Pi_t \otimes \Xi^{\frac{tg-d}{2}-n}$$
le systËme local d'Harris-Taylor associÈ o˘ 
$$\Xi:\frac{1}{2} \Zm \longrightarrow \overline \Zm_l^\times$$ 
est dÈfinie par $\Xi(\frac{1}{2})=q^{1/2}$ et
\begin{itemize}
\item $GL_{tg}(F_v)$ agit diagonalement sur $\Pi_t$ et sur $\LC(\pi_v[t]_D)_{\overline{1_{tg}}} 
\otimes \Xi^{\frac{tg-d}{2}-n}$ via son quotient $GL_{tg}(F_v) \twoheadrightarrow \Zm$,

\item le groupe de Weil $W_v$ en $v$ agit sur le facteur
$\Xi^{\frac{tg-d}{2}-n}$ via l'application $\deg: W_v \twoheadrightarrow \Zm$ qui 
envoie les frobenius gÈomÈtriques sur $1$.
\end{itemize}
On note aussi
$$HT_{\overline{1_{tg}}}(\pi_v,\Pi_t):=\widetilde{HT}_{\overline{1_{tg}}}(\pi_v,\Pi_t)[d-tg]$$
et
$$P(t,\pi_v)_{\overline{1_{tg}}}:=j^{=tg}_{\overline{1_{tg}},!*} HT_{\overline{1_{tg}}}
(\pi_v,\st_t(\pi_v)) \otimes \Lm(\pi_v)$$
le faisceau pervers d'Harris-Taylor associÈ o˘ $\Lm^\vee$ est la correspondance Langlands sur $F_v$.
En ce qui concerne les versions induites on les notera sans l'indice $\overline{1_{tg}}$, i.e.
$HT(\pi_v,\Pi_t)$ et $P(t,\pi_v)$.
\end{nota}

\rem On rappelle que $\pi'_v$ est inertiellement Èquivalente ‡ $\pi_v$ si et seulement
s'il existe un caractËre $\zeta: \Zm \longrightarrow  \overline \Qm_l^\times$ tel que 
$$\pi'_v \simeq \pi_v \otimes (\zeta \circ \val \circ \det).$$ 
Les faisceau pervers $P(t,\pi_v)$ 
ne dÈpend que de la classe d'Èquivalence inertielle de $\pi_v$ i.e. pour tout
$\chi:\Zm \longrightarrow \overline \Qm_l^\times$, 
les faisceaux pervers $P(t,\pi_v)$ et $P(t,\pi_v \otimes \chi)$,
munis de leurs actions par correspondances, sont isomorphes. En particulier, en notant 
$e_{\pi_v}$ la longueur de la restriction de $\pi_v[1]_D$ ‡ $\DC_{v,g}^\times$, on a
$$P(t,\pi_v)=e_{\pi_v} \PC(t,\pi_v)$$
o˘ $\PC(t,\pi_v)$ est un faisceau pervers Èquivariant simple. De la mÍme faÁon, il existe un
systËme local $\HT(\pi_v,\Pi_t)$ tel que $HT(\pi_v,\Pi_t)=e_{\pi_v} \HT(\pi_v,\Pi_t)$.

\begin{defi}
Pour tout $I \in \IC$, le faisceaux pervers des cycles Èvanescents
$R\Psi_{\eta_v,I}(\overline \Qm_l[d-1])(\frac{d-1}{2})$ sur $X_{I,\bar s_v}$ sera notÈ $\Psi_{I}$.
Le faisceau pervers de Hecke associÈ sur $X_{\IC,\bar s_v}$ est notÈ
$\Psi_{\IC}$.
\end{defi}
%

Dans \cite{boyer-invent2}, on dÈcompose dans le groupe de Grothendieck des faisceaux pervers
Èquivariants sur $X_{\IC,\bar s_v}$ relativement ‡ l'action de Hecke et du groupe de Weil en $v$,
le faisceau pervers $\Psi_{\IC}$ en termes des faisceaux pervers d'Harris-Taylor
$\PC(t,\pi_v)(n)$.

\begin{nota} Pour $K^\bullet$ un complexe de faisceaux, on notera $\hi^i K^\bullet$ son $i$-Ëme
faisceau de cohomologie.
\end{nota}

\section{Deux reprÈsentations induites du groupe mirabolique}
\label{para-mira}

Le nouvel argument pour calculer les faisceaux de cohomologie des faisceaux pervers d'Harris-Taylor 
et du complexe des cycles Èvanescents, repose sur deux lemmes, le premier 
\ref{lem-mirabolique} de thÈorie
des reprÈsentations du groupe mirabolique et le deuxiËme \ref{lem-jaffine} de nature gÈomÈtrique.

\begin{defi}
Soit $P=MN$ un parabolique standard de $GL_n$ de LÈvi $M$ et de radical unipotent $N$.
On note $\delta_P:P(F_v) \rightarrow \overline \Qm_l^\times$ l'application dÈfinie par
$\delta_P(h)=|\det (\ad(h)_{|\lie N})|^{-1}.$
Pour $(\pi_1,V_1)$ et $(\pi_2,V_2)$ des reprÈsentations de respectivement $GL_{n_1}(F_v)$ 
et $GL_{n_2}(F_v)$, et $P_{n_1,n_1+n_2}$ le parabolique standard de $GL_{n_1+n_2}$ de Levi 
$M=GL_{n_1} \times GL_{n_2}$ et de radical unipotent $N$, 
$$\pi_1 \times \pi_2$$
dÈsigne l'induite parabolique normalisÈe de $P_{n_1,n_1+n_2}(F_v)$ ‡ $GL_{n_1+n_2}(F_v)$ de 
$\pi_1 \otimes \pi_2$ c'est ‡ dire
l'espace des fonctions $f:GL_{n_1+n_2}(F_v) \rightarrow V_1 \otimes V_2$ telles que
$$f(nmg)=\delta_{P_{n_1,n_1+n_2}}^{-1/2}(m) (\pi_1 \otimes \pi_2)(m) \Bigl ( f(g) \Bigr ),
\quad \forall n \in N, ~\forall m \in M(F_v), ~ \forall g \in GL_{n_1+n_2}(F_v),$$
ou encore en terme d'induite classique
\addtocounter{smfthm}{1}
\begin{equation} \label{eq-torsion}
\pi_1 \times \pi_2 \simeq \ind_{P_{n_1,n_1+n_2}(F_v)}^{GL_{n_1+n_2}(F_v)} \pi_1\{\frac{n_2}{2} \}
\otimes \pi_2 \{ - \frac{n_1}{2} \}.
\end{equation}
\end{defi}

\rem Rappelons qu'une reprÈsentation $\pi$ de $GL_n(F_v)$ est dite \textit{cuspidale} si elle n'est pas 
un sous-quotient d'une induite parabolique propre. 

\begin{notas} (cf. \cite{zelevinski2} \S 2)
Soient $g$ un diviseur de $d=sg$ et $\pi_v$ une reprÈsentation cuspidale
irrÈductible de $GL_g(F_v)$. 
\begin{itemize}
\item L'unique quotient (resp. sous-reprÈsentation) irrÈductible de
$$\pi_v\{ \frac{1-s}{2} \} \times \pi_v\{\frac{3-s}{2} \} \times \cdots \times \pi_v \{ \frac{s-1}{2} \}$$
est notÈ $\st_s(\pi_v)$ (resp. $\speh_s(\pi_v)$).

\item L'unique sous-espace irrÈductible de 
$$\st_t(\pi_v \{ \frac{-s}{2} \} ) \times \speh_s (\pi_v \{ \frac{t}{2} \} )$$ 
est notÈ $LT_{\pi_v}(t-1,s)$.
\end{itemize}
\end{notas}

Notons $M_n(F_v)$ le sous-groupe mirabolique de $GL_n(F_v)$
dÈfini comme  le sous-groupe de $P_{1,n}(F_v)$ dont le premier coefficient en haut ‡ gauche
est Ègal ‡ $1$. Suivant \cite{zelevinski1}
une reprÈsentation irrÈductible $\tau$ de $M_n(F)$ admet une unique dÈrivÈe non nulle, disons d'ordre $k$, 
qui est alors une reprÈsentation irrÈductible $\tau^{(k)}$ de $GL_{n-k}(F_v)$: on peut en outre
reconstituer $\tau$ ‡ partir de cette dÈrivÈe.

\rem Dans \cite{zelevinski1}, les auteurs utilisent plutÙt le groupe obtenu ‡ partir de $M_n(F_v)$
en tordant les actions par $g \mapsto \sigma (\lexp t g^{-1} ) \sigma^{-1}$, 
o˘ $\sigma$ est la matrice de permutation associÈ au cycle $(12\cdots n)$. 
Le calcul des dÈrivÈes de \cite{zelevinski2} theorem 3.5, proposition 9.6 est alors inversÈ, cf. 
les exemples ci-aprËs.

Pour Ètudier les reprÈsentations de $GL_n(F)$
et notamment dÈcrire les sous-quotients irrÈductibles des induites paraboliques comme ci-avant,
le principe est alors de les restreindre au groupe mirabolique et de regarder leurs dÈrivÈes.
On a en particulier les calculs de dÈrivÈes suivant.
\begin{itemize}
\item On a une formule de Leibnitz: $(\pi_1 \times \pi_2)^{(k)}$ admet une filtration dont les graduÈs
sont les $(\pi_1)^{(i)} \times (\pi_2)^{(k-i)}$ pour $i=0,\cdots,k$. 

\item $\speh_s(\pi_v)$ possËde une unique dÈrivÈe non nulle d'ordre $g$ isomorphe ‡
$\speh_{s-1}(\pi_v \{ \frac{1}{2} \})$.

\item Les dÈrivÈes d'ordre $k$ de $\st_s(\pi_v)$ sont nulles si $k$ n'est pas de la forme $tg$ pour
$1 \leq t \leq s$. Pour $k=tg$ avec $1 \leq t < s$, cette dÈrivÈe est
isomorphe ‡ $\st_{s-t}(\pi_v \{ -\frac{t}{2} \})$.

\item Les dÈrivÈes d'ordre $k$ de $LT_{\pi_v}(t-1,s)$ sont nulles si $k$ n'est pas de la forme $\delta g$
pour $1 \leq k \leq t$. Pour un tel $k=\delta g$, elle est isomorphe ‡ la reprÈsentation induite irrÈductible
$\st_{t-\delta}(\pi_v \{ -\frac{\delta+s}{2} \} ) \times \speh_s(\pi_v \{ \frac{t}{2} \} )$.
%
\end{itemize}

\rem Pour $LT_{\pi_v}(t,s-1)$, on trouve 
$$LT_{\pi_v}(t,s-1)^{(\delta g)} = \
\st_{t+1-\delta}(\pi_v \{ -\frac{\delta+s-1}{2} \} ) \times \speh_{s-1}(\pi_v \{ \frac{t+1}{2} \} ).$$
%
%
Le support cuspidal de $LT_{\pi_v}(t-1,s)^{(\delta g)}$ est ainsi
\begin{multline*}
\pi_v \{ \frac{1-s-t}{2} \}, \pi_v \{ \frac{3-s-t}{2} \}, \cdots, \pi_v \{ \frac{-1-s+t-2 \delta}{2} \}, \\
\pi_v \{ \frac{1-s+t}{2} \}, \pi_v \{ \frac{3-s+t}{2} \},\cdots, \pi_v \{ \frac{s+t-1}{2} \},
\end{multline*}
alors que celui de $LT_{\pi_v}(t,s-1)^{(\delta g)}$ est
\begin{multline*}
\pi_v \{ \frac{1-s-t}{2} \}, \pi_v \{ \frac{3-s-t}{2} \}, \cdots, \pi_v \{ \frac{1-s+t-2 \delta}{2} \}, \\
\pi_v \{ \frac{3-s+t}{2} \}, \pi_v \{ \frac{5-s+t}{2} \},\cdots, \pi_v \{ \frac{s+t-1}{2} \}.
\end{multline*}
On notera en particulier que ces supports sont distincts et que donc $LT_{\pi_v}(t-1,s)$
et  $LT_{\pi_v}(t,s-1)$ ont des dÈrivÈes irrÈductibles distinctes ou nulles.

\begin{lemm} \label{lem-mirabolique}
Soit $\pi_v$ une reprÈsentation irrÈductible cuspidale de $GL_g(F_v)$, alors en tant que reprÈsentation
du parabolique $P_{1,(t+s)g}(F_v)$, on a des isomorphismes
$$\st_t(\pi_v\{ -\frac{s}{2} \} )_{|M_{tg}(F_v)} \times \speh_s (\pi_v \{ \frac{t}{2} \} ) \simeq
LT_{\pi_v} (t-1,s)_{|M_{(t+s)g}(F_v)},$$
et
$$\st_t(\pi_v\{ -\frac{s}{2} \} ) \times \speh_s (\pi_v \{ \frac{t}{2} \} )_{|M_{sg}(F_v)} \simeq 
LT_{\pi_v} (t,s-1)_{|M_{(t+s)g}(F_v)},$$
o˘ dans le premier isomorphisme, l'induite parabolique est relativement ‡ 
$$\left ( \begin{array}{cc} M_{tg} & U \\ 0 & GL_{sg} \end{array} \right )$$
alors que dans le deuxiËme, il s'agit de l'induite ‡ support compact relativement ‡
$$\left ( \begin{array}{ccc} 1 & 0 & V_{sg-1} \\ 0 & GL_{tg} & U \\ 0 & 0 & GL_{sg-1} \end{array} \right ),$$
o˘ on note $V_{n-1}(F_v) \simeq F_v^{n-1}$ le radical unipotent de $M_n(F_v)$.
\end{lemm}


\begin{proof} 
On rappelle cf. par exemple \cite{vigneras-livre} \S III.1.10, que
\begin{multline*}
0 \rightarrow 
\st_t(\pi_v\{ -\frac{s}{2} \} ) \times \speh_s (\pi_v \{ \frac{t}{2} \} )_{|M_{sg}(F_v)} \\
\longrightarrow
\Bigl ( \st_t(\pi_v\{ -\frac{s}{2} \} )  \times \speh_s (\pi_v \{ \frac{t}{2} \} ) \Bigr )_{|M_{(t+s)g}(F_v)} \\
\longrightarrow
 \st_t(\pi_v\{ -\frac{s}{2} \} )_{|M_{tg}(F_v)} \times \speh_s (\pi_v \{ \frac{t}{2} \} )
\rightarrow 0,
\end{multline*}
o˘ le deuxiËme terme est l'induite parabolique relativement ‡ 
$\left ( \begin{array}{cc} M_{tg} & U \\ 0 & GL_{sg} \end{array} \right )$
et le premier, l'induite ‡ support compact relativement ‡
$\left ( \begin{array}{ccc} 1 & 0 & V_{sg-1} \\ 0 & GL_{tg} & U \\ 0 & 0 & GL_{sg-1} \end{array} \right )$.
En outre pour tout $k \geq 0$, la dÈrivÈe d'ordre $k$ de 
$\st_t(\pi_v\{ -\frac{s}{2} \} )_{|M_{tg}(F_v)} \times \speh_s (\pi_v \{ \frac{t}{2} \} )$ est, cf. 
\cite{vigneras-livre} p153, donnÈe par celle de 
$\st_t(\pi_v \{ -\frac{s}{2} \} )_{|M_{tg}(F_v)}$ induite avec $\speh_s (\pi_v \{ \frac{t}{2} \} )$, i.e.
$$
\Bigl ( \st_t(\pi_v \{ -\frac{s}{2} \} )_{|M_{tg}(F_v)} \times \speh_s (\pi_v \{ \frac{t}{2} \} ) \Bigr )^{(k)} 
\simeq \Bigl ( \st_t(\pi_v \{ -\frac{s}{2} \})_{|M_{tg}(F_v)} \Bigr )^{(k)} \times 
\speh_s (\pi _v\{ \frac{t}{2} \} ).
$$
Ainsi on remarque que
$\st_t(\pi_v\{ -\frac{s}{2} \} )_{|M_{tg}(F_v)} \times \speh_s (\pi_v \{ \frac{t}{2} \} )$ et
$LT_{\pi_v}(t-1,s)_{|M_{(t+s)g}(F_v)} $ ont les mÍmes dÈrivÈes.
ConsidÈrons alors le morphisme composÈ
\begin{multline*}
LT_{\pi_v}(t-1,s)_{|M_{(t+s)g}(F_v)} \hookrightarrow 
\Bigl ( \st_t(\pi_v\{ -\frac{s}{2} \} )  \times \speh_s (\pi_v \{ \frac{t}{2} \} ) \Bigr )_{|M_{(t+s)g}(F_v)} \\
\twoheadrightarrow 
\st_t(\pi_v\{ -\frac{s}{2} \} )_{|M_{tg}(F_v)} \times \speh_s (\pi_v \{ \frac{t}{2} \} )
\end{multline*}
et notons 
$K \hookrightarrow \st_t(\pi_v \{ -\frac{s}{2} \} ) \times \speh_s (\pi_v \{ \frac{t}{2} \} )_{|M_{sg}(F_v)}$
son noyau. Par soustraction, les dÈrivÈes de $\st_t(\pi_v \{ -\frac{s}{2} \} ) 
\times \speh_s (\pi_v \{ \frac{t}{2} \} )_{|M_{sg}(F_v)}$
sont les mÍmes que celles de $LT_{\pi_v}(t,s-1)_{|M_{(t+s)g}(F_v)}$ et sont donc, d'aprËs la remarque prÈcÈdente,
distinctes de celles de $LT_{\pi_v}(t-1,s)_{|M_{(t+s)g}(F_v)}$ et donc des sous-quotients non nuls de $K$, 
de sorte que $K$ est nÈcessairement nul, i.e.
$$LT_{\pi_v}(t-1,s) \hookrightarrow \st_t(\pi\{ -\frac{s}{2} \} )_{|M_{tg}(F_v)} \times 
\speh_s (\pi \{ \frac{t}{2} \} ).$$
Comme ces deux termes ont les mÍmes dÈrivÈes, cette injection est un isomorphisme.
%
%
\end{proof}

\section{Stratification de Newton raffinÈe}

\begin{defis} \label{defi-strate2}
Pour tout $h_0 \leq h$, on introduit les strates \og non pures \fg{} suivantes
$$X^{\geq h}_{\IC,\bar s_v,\overline{1_{h_0}}}:= X^{\geq h}_{\IC,\bar s_v} \cap 
X^{\geq h_0}_{\IC,\bar s_v,\overline{1_{h_0}}},
$$
ainsi que
$$X^{\geq h}_{\IC,\bar s_v,\overline{1_{h_0}} \setminus \overline{1_{h_0+1}}}:=
X^{\geq h}_{\IC,\bar s_v,\overline{1_{h_0}}}
\setminus X^{\geq h'}_{\IC,\bar s_v,\overline{1_{h_0+1}}} \quad \hbox{ o˘ } 
\left \{ \begin{array}{ll} h'=h & \hbox{si } h_0< h \\ h'=h+1 & \hbox{sinon.} \end{array} \right.$$
On notera alors
$$j^{\geq h}_{\overline{1_{h_0}} \setminus \overline{1_{h_0+1}}}: 
X^{\geq h}_{\IC,\bar s_v,\overline{1_{h_0}} \setminus \overline{1_{h_0+1}}}
\hookrightarrow X^{\geq h}_{\IC,\bar s_v,\overline{1_{h_0}}}, \qquad
i^{h \leq +1}_{\overline{1_{h_0}},\overline{1_{h_0+1}}}:X^{\geq h+1}_{\IC,\bar s_v,\overline{1_{h_0+1}}}
\hookrightarrow X^{\geq h}_{\IC,\bar s_v,\overline{1_{h_0}}},$$
et $j^{=h}_{\overline{1_{h_0}} \setminus \overline{1_{h_0+1}}}:=i^h_{\overline{1_{h_0}}} \circ 
j^{\geq h}_{\overline{1_{h_0}} \setminus \overline{1_{h_0+1}}}$.
Dans le cas $h_0=h$, on notera simplement 
$$X^{\geq h}_{\IC,\bar s_v,\neq 1}:=X^{\geq h}_{\IC,\bar s_v,\overline{1_h} \setminus \overline{1_{h+1}}}$$ 
ainsi que 
$j^{=h}_{\neq 1}:=j^{=h}_{\overline{1_h} \setminus \overline{1_{h+1}}}$.
On utilisera aussi la notation
$$i^{h+1 \leq +0}_{\overline{1_{h_0}},\overline{1_{h_0+1}}}:
X^{\geq h+1}_{\IC,\bar s_v,\overline{1_{h_0+1}}} \hookrightarrow 
X^{\geq h+1}_{\IC,\bar s_v,\overline{1_{h_0}}}.$$
\end{defis}

\begin{lemm} \label{lem-jaffine} \phantomsection
L'inclusion ouverte $j^{\geq h}_{\neq 1}$ est affine.
\end{lemm}

\begin{proof} 
Soit $\GC(h)$ le groupe de Barsotti-Tate universel sur $X_{I,\bar s_v,\overline{1_h}}^{=h}$:
$$0 \rightarrow \GC(h)^c \longrightarrow \GC(h) \longrightarrow \GC(h)^{et} \rightarrow 0$$
o˘ $\GC(h)^c$ (resp. $\GC(h)^{et}$) est connexe (resp. Ètale) de dimension $h$ (resp. $d-h$).
Notons 
$$\iota_{m_1}:(\PC_v^{-m_1}/\OC_v)^d \longrightarrow \GC(h)[p^{m_1}]$$
la structure de niveau universelle. Notant $(e_i)_{1 \leq i \leq d}$ la base canonique de
$(\PC_v^{-m_1}/\OC_v)^d$, on rappelle alors que $X_{I,\bar s_v,\overline{1_h}}^{=h}$ est 
dÈfini par la propriÈtÈ que $\bigl \{ \iota_{m_1}(e_i):~1 \leq i \leq h \bigr \}$ forme une base de Drinfeld de 
$\GC(h)^c[p^{m_1}]$.

Ainsi pour tout $\spec A \longrightarrow X^{\geq h}_{\IC,\bar s_v,\overline{1_h}}$, le fermÈ 
$X^{\geq h+1}_{\IC,\bar s_v,\overline{1_{h+1}}} \times_{X^{\geq h}_{\IC,\bar s_v,\overline{1_h}}} 
\spec A$ est donnÈ par l'annulation de $\iota(e_{h+1})$, d'o˘ le rÈsultat.
\end{proof}

\begin{nota} Pour $h_0 < h$ et un systËme local d'Harris-Taylor $HT(\pi_v,\Pi_t)$ sur
$X^{=h}_{\IC,\bar s_v}$, avec donc $h=tg$, on notera 
$$HT_{\neq \overline{1_{h_0}}}(\pi_v,\Pi_t)$$
la restriction de $HT(\pi_v,\Pi_t)$ ‡ $X^{=h}_{\IC,\bar s_v, \neq \overline{1_{h_0}}}$. 
Lorsque $h=h_0$, on notera plus simplement $HT_{\neq 1}(\pi_v,\Pi_t)$.
\end{nota}

On introduit de mÍme $X^{=h}_{\IC,\bar s_v,\overline{1_{h_0}}}:=X^{=h}_{\IC,\bar s_v} \cap 
X^{\geq h_0}_{\IC,\bar s_v,\overline{1_{h_0}}}$ ainsi que son adhÈrence
$X^{\geq h}_{\IC,\bar s_v,\overline{1_{h_0}}}$ dans $X^{\geq h}_{\IC,\bar s_v}$
avec
$$j^{\geq h}_{\overline{1_{h_0}}}: X^{=h}_{\IC,\bar s_v,\overline{1_{h_0}}} \hookrightarrow
X^{\geq h}_{\IC,\bar s_v,\overline{1_{h_0}}}$$
et $j^{=h}_{\overline{1_{h_0}}}=i^h_{\overline{1_{h_0}}} \circ j^{\geq h}_{\overline{1_{h_0}}}$
o˘ $i^h_{\overline{1_{h_0}}}: X^{\geq h}_{\IC,\bar s_v,\overline{1_{h_0}}} \hookrightarrow
X_{\IC,\bar s_v}$. On notera aussi, cf. (\ref{eq-strate-induite})
$$HT_{\overline{1_{h_0}}}(\pi_v,\Pi_t) = HT_{\overline{1_{h}}}(\pi_v,\Pi_t) \times_{P_{h_0,h,d}(F_v)}
P_{h,d}(F_v)$$ 
la restriction de $HT(\pi_v,\Pi_t)$ ‡ $X^{=h}_{\IC,\bar s_v,\overline{1_{h_0}}}$.

Fixons ‡ prÈsent une reprÈsentation irrÈductible cuspidale $\pi_v$ de $GL_g(F_v)$ et notons
$s=\lfloor \frac{d}{g} \rfloor$.
Rappelons l'ÈnoncÈ suivant de \cite{boyer-invent2} 4.5.1 dans le groupe de Grothendieck des faisceaux 
pervers de Hecke\footnote{pour les torsions, cf. la formule (\ref{eq-torsion})}
\addtocounter{smfthm}{1}
\begin{multline} \label{eq00}
j^{=tg}_{\overline{1_{tg}},!} HT_{\overline{1_{tg}}} (\pi_v,\Pi_t) =  \\
\sum_{i=t}^{s} \lexp p j^{=ig}_{\overline{1_{tg}},!*} HT_{\overline{1_{tg}}} \bigl ( \pi_v,\Pi_t \{ (i-t)\frac{g-1}{2} \} 
\otimes \st_{i-t}(\pi_v \{ -t \frac{g-1}{2} \} ) \bigr ) (\frac{i-t}{2}).
\end{multline}
Dans \cite{boyer-invent2}, pour calculer les germes des faisceaux de cohomologie de 
$\lexp p j^{=tg}_{\overline{1_{tg}},!*} HT_{\overline{1_{tg}}}(\pi_v,\Pi_t)$, on procËde comme suit.

(a) On considËre tout d'abord la filtration par les poids de $j^{=tg}_{\overline{1_{tg}},!} HT_{\overline{1_{tg}}}(\pi_v,\Pi_t)$, soit
$$\Fil^{t-s} (\pi_v,\Pi_t) \subset \cdots \subset \Fil^{-1}(\pi_v,\Pi_t) \subset 
\Fil^{0}(\pi_v,\Pi_t) =j^{=tg}_{\overline{1_{tg}},!} HT_{\overline{1_{tg}}}(\pi_v,\Pi_t)$$ 
dont les graduÈs sont 
$$\gr^{-k}(\pi_v,\Pi_t) \simeq \lexp p j^{=(t+k)g}_{\overline{1_{tg}},!*} HT_{\overline{1_{tg}}}
(\pi_v,\Pi_t \{ k \frac{g-1}{2} \} \otimes \st_{k}(\pi_v \{-t\frac{g-1}{2} \} )) (-\frac{k}{2}) \bigr ).$$

(b) On dispose alors d'une suite spectrale faisceautique
\addtocounter{smfthm}{1}
\begin{equation} \label{eq-ss1}
E_1^{p,q}=\hi^{p+q} \gr^{-p}(\pi_v,\Pi_t) \Rightarrow \hi^{p+q} j^{=tg}_{1,!} HT_1(\pi_v,\Pi_t).
\end{equation}
Pour tout point gÈomÈtrique $z$ de $X^{=h}_{\IC,\bar s_v}$, la fibre en $z$ de l'aboutissement
de cette suite spectrale est connue, nulle dËs que $h>tg$. En raisonnant par rÈcurrence les
fibres en $z$ des $E_1^{p,q}$ sont aussi connues pour tout $(p,q)$ si $p>0$:
\begin{itemize}
\item si $h$ n'est pas de la forme $(t+\delta)g$ alors les fibres en $z$ de $E_1^{p,q}$ pour $p>0$
sont toutes nulles;

\item pour $h=(t+\delta)g$, elles sont nulles sauf pour $q=(t+\delta)g-d-\delta$.
\end{itemize}

(c) Ainsi pour calculer les fibres des faisceaux de cohomologie de $j^{=tg}_{1,!*} HT_1(\pi_v,\Pi_t)$,
on est ramenÈ ‡ Ètudier les flËches $d_1^{p,q}$ pour $q=(t+\delta)g-d-\delta$ et $p>0$, et ‡ montrer
qu'elles sont non nulles ce qui suffit ‡ les caractÈriser. 

Le but des paragraphes suivant est de proposer une nouvelle preuve ÈlÈmentaire utilisant simplement 
les applications $j^{=h}_{\neq 1}$ et la thÈorie des reprÈsentations du groupe mirabolique.

\section{Faisceaux de cohomologie des faisceaux pervers d'Harris-Taylor}

CommenÁons par poser $h:=tg$ et
$$P(\pi_v,\Pi_t):=i^{h+1}_* \lexp p \hi^{-1} i^{h+1,*}_{\overline{1_h}} \bigl ( \lexp p j^{=h}_{\overline{1_h},!*}
HT_{\overline{1_h}}(\pi_v,\Pi_t) \bigr ),$$ 
de sorte que
\addtocounter{smfthm}{1}
\begin{equation} \label{eq-sec0}
0 \rightarrow P(\pi_v,\Pi_t) \longrightarrow j^{=h}_{\overline{1_h},!} HT_{\overline{1_h}}(\pi_v,\Pi_t)
\longrightarrow \lexp p j^{=h}_{\overline{1_h},!*} HT_{\overline{1_h}}(\pi_v,\Pi_t) \rightarrow 0.
\end{equation}

\begin{lemm} \label{lem-HT1}
Le complexe $i^{h+1,*}_{\overline{1_{h+1}}} P(\pi_v,\Pi_t)$ est pervers.
\end{lemm}

\begin{proof}
Suivant la notation \ref{nota-jh1},
notons $F:=\lexp p j^{\geq h}_{\overline{1_h},!*} HT_{\overline{1_h}}(\pi_v,\Pi_t)$ de sorte qu'il s'agit 
de montrer que 
$$i^{h+1 \leq +0,*}_{\overline{1_h},\overline{1_{h+1}}} \bigl ( \lexp p \hi^{-1} i^{h\leq +1,*}_{\overline{1_h}} F \bigr )$$
est pervers. Pour ce faire, on utilise la suite spectrale
$$E_2^{r,s}=\lexp p \hi^r i^{h+1 \leq +0,*}_{\overline{1_h},\overline{1_{h+1}}} \Bigl ( \lexp p \hi^s 
i^{h \leq +1,*}_{\overline{1_h}} F \Bigr ) \Rightarrow \lexp p \hi^{r+s} 
i^{h \leq +1,*}_{\overline{1_h},\overline{1_{h+1}}} F.$$
ConsidÈrons le triangle distinguÈ
$$j^{\geq h}_{\overline{1_h},!} j^{\geq h,*}_{\overline{1_h}} F \longrightarrow F \longrightarrow 
i^{h \leq +1}_{\overline{1_h},*} i^{h \leq +1,*}_{\overline{1_h}} F \leadsto.$$
Comme $j^{\geq h}_{\overline{1_h}}$ est affine, alors $j^{\geq h}_{\overline{1_h},!} j^{\geq h,*}_{\overline{1_h}} F$
est pervers et  la suite exacte longue de cohomologie perverse
du triangle distinguÈ prÈcÈdent s'Ècrit
$$0 \rightarrow i^{h \leq +1}_{\overline{1_h},*} \lexp p \hi^{-1} i^{h \leq +1,*}_{\overline{1_h}} F \longrightarrow 
\lexp p j^{\geq h}_{\overline{1_h},!} j^{\geq h,*}_{\overline{1_h}} F \longrightarrow F \longrightarrow 
i^{h \leq +1}_{\overline{1_h},*} \lexp p \hi^0 i^{h \leq +1,*}_{\overline{1_h}} F \rightarrow 0,$$
et donc en particulier, $\lexp p \hi^s i^{h \leq +1,*}_{\overline{1_h}} F$ est nul
pour tout $s<-1$. La surjectivitÈ $j^{\geq h}_{\overline{1_h},!} j^{\geq h,*}_{\overline{1_h}} F
\twoheadrightarrow F$, implique aussi la nullitÈ pour $s=0$. Ainsi la suite spectrale prÈcÈdente
dÈgÈnËre en $E_2$ avec
$$\lexp p \hi^r i^{h \leq +1,*}_{\overline{1_h},\overline{1_{h+1}}} F \simeq 
\lexp p \hi^{r+1} i^{h+1 \leq +0,*}_{\overline{1_h},\overline{1_{h+1}}}
\bigl ( \lexp p \hi^{-1} i^{h \leq +1,*}_{\overline{1_h}} F \bigr ).$$

De la mÍme faÁon, comme $j^{\geq h}_{\neq 1}$ est affine,
$\lexp p \hi^r i^{h \leq +1,*}_{\overline{1_h},\overline{1_{h+1}}} F $ est nul pour $r<-1$, d'o˘ le rÈsultat.
\end{proof}

En particulier, on a donc une suite exacte courte de faisceaux pervers
\addtocounter{smfthm}{1}
\begin{equation} \label{eq-sec1}
0 \rightarrow j^{=h+1}_{\overline{1_h} \setminus \overline{1_{h+1}},!}
 j^{=h+1,*}_{\overline{1_h} \setminus \overline{1_{h+1}} } P(\pi_v,\Pi_t)  
\longrightarrow P(\pi_v,\Pi_t) \longrightarrow i^{h+1}_{\overline{1_{h+1}},*} 
\lexp p i^{h+1,*}_{\overline{1_{h+1}}} P(\pi_v,\Pi_t) \rightarrow 0.
\end{equation}

\begin{prop} (cf. aussi le lemme B.3.3 de \cite{boyer-duke}) \label{prop-fond1} \\
Avec les notations prÈcÈdentes on a
$$ i^{tg+1}_{\overline{1_{tg+1}},*} \lexp p i^{tg+1,*}_{\overline{1_{tg+1}}} P(\pi_v,\Pi_t) \simeq
\lexp p j^{=(t+1)g}_{\overline{1_{tg+1}},!*} 
HT_{\overline{1_{tg+1}}} (\pi_v,\Pi_t\{ \frac{g-1}{2} \} \otimes (\pi_v)_{|P_{1,g}(F_v)} \{ -t \frac{g-1}{2} \} ) 
(\frac{1}{2}).$$
\end{prop}

\rem Rappelons que le terme de droite de l'isomorphisme de la proposition prÈcÈdente est une notation
pour
$$\ind_{P_{tg,tg+1,(t+1)g,d}(F_v)}^{P_{tg,tg+1,d}(F_v)} \lexp p j^{=(t+1)g}_{\overline{1_{(t+1)g}},!*} 
HT_{\overline{1_{(t+1)g}}} (\pi_v,\Pi_t\{ \frac{g-1}{2} \} \otimes (\pi_v)_{|P_{1,g}(F_v)} \{ -t \frac{g-1}{2} \} ) 
(\frac{1}{2}).$$

Avant de donner la preuve de cette proposition, commenÁons par deux corollaires.

\begin{coro} \label{coro-hi}
Soient $h_0 \geq h=tg$ et $z$ un point gÈomÈtrique de $X^{=h_0}_{\IC,\bar s_v,\overline{1_{h}}}$. Le
germe en $z$ du $i$-Ëme faisceau de cohomologie de $\lexp p j^{=tg}_{\overline1_{h},!*} 
HT_{\overline{1_{h}}}(\pi_v,\st_t(\pi_v))$ est
\begin{itemize}
\item nul si $(h_0,i)$ n'est pas de la forme $\bigl ( (t+\delta)g, (t+\delta)g-d-\delta \bigr )$ avec
$(t+\delta)g \leq d$,

\item et sinon, muni des actions par correspondances de
$G(\Am^{\oo,v}) \times P_{h,d}(F_v) \times \Zm$, il est isomorphe au germe en $z$ de 
$HT_{\overline{1_h}}(\pi_v,\st_t(\pi_v \{ \delta \frac{g-1}{2} \} ) \otimes \speh_\delta 
(\pi_v \{-t \frac{g-1}{2} \} ))$.
\end{itemize}
\end{coro}

\begin{proof}
On raisonne par rÈcurrence sur $t$ de $s$ ‡ $1$. Pour $t=s$, on a 
$\lexp p j^{=tg}_{\overline{1_{h}},!*} HT_{\overline{1_{h}}} (\pi_v,\st_t(\pi_v)) \simeq 
\lexp p j^{=tg}_{\overline{1_{h}},!} HT_{\overline{1_{h}}}(\pi_v,\st_t(\pi_v))$ 
d'o˘ le rÈsultat. Supposons alors le rÈsultat acquis au rang $t+1$. 

\rem De la suite spectrale associÈe ‡ la filtration par les poids de 
$ j^{=tg}_{\overline{1_{h}},!} HT_{\overline{1_{h}}}(\pi_v,\st_t(\pi_v))$ dont les graduÈs sont donnÈs
par (\ref{eq00}), les germes des faisceaux de cohomologie de 
$\lexp p j^{=tg}_{\overline{1_{h}},!*} HT_{\overline{1_{h}}}(\pi_v,\st_t(\pi_v))$ sont de la forme
$HT_{\overline{1_{h}}}(\pi_v,\st_t(\pi_v \{\delta \frac{g-1}{2} \}) \otimes \Pi)$ o˘ $\Pi$ est une reprÈsentation de 
$GL_{\delta g}(F_v)$ dont le support cuspidal est un segment de Zelevinski centrÈ sur 
$\pi_v \{ -t \frac{g-1}{2})$.

Par Èquivariance, on supposera
que $z$ est un point de $X^{\geq tg}_{\IC,\bar s_v,\overline{1_{tg+1}}}$.
D'aprËs (\ref{eq-sec0}),  le germe en $z$ du $i$-Ëme faisceau de cohomologie de 
$\lexp p j^{=tg}_{\overline{1_{h}},!*} HT_{\overline{1_{h}}}(\pi_v,\st_t(\pi_v))$, muni des actions 
par correspondances de $G(\Am^{\oo,v}) \times P_{h,d}(F_v) \times \Zm$, est isomorphe ‡ celui du 
$(i+1)$-Ëme de $P(\pi_v,\st_t(\pi_v))$, lequel d'aprËs (\ref{eq-sec1}) et la proposition prÈcÈdente est,
muni des actions par correspondances de $G(\Am^{\oo,v}) \times P_{h,h+1,d}(F_v) \times \Zm$,
isomorphe ‡ celui du $(i+1)$-Ëme faisceau de cohomologie de 
$\lexp p j^{=(t+1)g}_{\overline{1_{tg+1}},!*} 
HT_{\overline{1_{tg+1}}} (\pi_v,\st_t(\pi_v \{ \frac{g-1}{2} \}) \otimes 
(\pi_v)_{|P_{1,g}(F_v)} \{ -t \frac{g-1}{2} \} ) (\frac{1}{2})$.

D'aprËs l'hypothËse de rÈcurrence, ce germe est nul si $(h_0,i+1)$ n'est pas de la forme \\
$\bigl ( (t+1+\delta-1)g,(t+1+\delta-1)g-d-\delta+1 \bigr )$ et sinon il est isomorphe,
muni des actions par correspondances de $G(\Am^{\oo,v}) \times P_{h,h+1,d}(F_v) \times \Zm$,
au germe en $z$ de 
$HT_{\overline{1_{(t+\delta)g}}}(\pi_v, \Pi) (\frac{1}{2}) \times_{P_{tg,tg+1,(t+\delta)g,d}(F_v)}
P_{tg,tg+1,d}(F_v)$ o˘
$$\Pi:=
\st_t(\pi_v \{\delta \frac{g-1}{2} \}) \otimes 
(\pi_v)_{|P_{1,g}(F_v)} \{ (\delta-1-t) \frac{g-1}{2} \}  \otimes \speh_{\delta-1}
(\pi_v \{ - (t+1) \frac{g-1}{2} \}).$$
En prenant $t=1$ et $s=\delta-1$ dans le premier isomorphisme du lemme \ref{lem-mirabolique}, on
obtient alors 
$$(\pi_v)_{|M_{g}(F_v)} \{ (\delta-1) \frac{g-1}{2} \}  \otimes \speh_{\delta-1}
(\pi_v \{ - \frac{g-1}{2} \}) \simeq \speh_\delta(\pi_v)_{|M_{\delta g}(F_v)}$$
de sorte que $\Pi \simeq \st_t(\pi_v \{\delta \frac{g-1}{2} \}) \otimes 
\speh_\delta(\pi_v \{ -t \frac{g-1}{2}))_{|M_{\delta g}(F_v)}$. 
N'ayant qu'une unique dÈrivÈe non nulle irrÈductible,
$\speh_\delta(\pi_v \{ -t \frac{g-1}{2}))_{|M_{\delta g}(F_v)}$ est irrÈductible de sorte que 
le rÈsultat dÈcoule de la remarque ci-avant.
\end{proof}

\rem Une autre faÁon d'Ènoncer le rÈsultat prÈcÈdent consiste ‡ dire que les flËches de la suite
spectrale (\ref{eq-ss1}) sont \og le moins triviales possible \fg.

\begin{coro} \label{coro-j}
Soit $h_0=t_0g \leq h=tg$ avec $1 \leq t <s$. On a alors la suite exacte courte
\begin{multline*}
0 \rightarrow 
\lexp p j^{=(t+1)g}_{\overline{1_{h_0+1}},!*} HT_{\overline{1_{h_0+1}}}
\bigr (\pi_v, \Pi_{t_0} \{ \frac{g-1}{2} \} \otimes \st_{t+1-t_0}
(\pi_v \{ -t_0\frac{g-1}{2} \} )_{P_{1,(t+1-t_0)g}(F_v)} \bigl ) (\frac{1}{2}) \\
\longrightarrow j^{=tg}_{\overline{1_{h_0}} \setminus \overline{1_{h_0+1}},!}
j^{=tg,*}_{\overline{1_{h_0}} \setminus \overline{1_{h_0+1}}} 
\bigl ( \lexp p j^{=tg}_{\overline{1_{h_0}}!*} HT_{\overline{1_{h_0}}}(\pi_v,\Pi_{t_0} \otimes 
\st_{t-t_0}(\pi_v  \{ -t_0\frac{g-1}{2} \}) ) \bigr ) \\
\longrightarrow \lexp p j^{=tg}_{\overline{1_{h_0}} \setminus \overline{1_{h_0+1}},!*} 
HT_{\overline{1_{h_0}} \setminus \overline{1_{h_0+1}}} (\pi_v,\Pi_{t_0} \otimes 
\st_{t-t_0}(\pi_v  \{ -t_0\frac{g-1}{2} \}))  \rightarrow 0.
\end{multline*}
\end{coro}

\rem En ce qui concerne les dÈcalages, et notamment o˘ est passÈ le $\frac{g-1}{2}$ 
qui s'applique sur $\Pi_{t_0} \otimes \st_{t-t_0}(\pi_v \{ -t_0\frac{g-1}{2} \})$, 
rappelons, cf. (\ref{eq-torsion}), que 
$$\ind_{P_{(t-t_0)g,(t-t_0+1)g}(F_v)}^{GL_{(t-t_0+1)g}(F_v)} \st_{t-t_0}(\pi_v \{ \frac{g-1}{2} \}) \otimes
\pi_v \{ -(t-t_0) \frac{g-1}{2} \} \twoheadrightarrow \st_{t-t_0+1}(\pi_v).$$

\begin{proof}
Comme 
\begin{multline*}
\lexp p j^{=tg}_{\overline{1_{h_0}} \setminus \overline{1_{h_0+1}},!*} 
j^{=tg,*}_{\overline{1_{h_0}} \setminus \overline{1_{h_0+1}}}  \bigl ( \lexp p j^{=tg}_{\overline{1_{h_0}}!*} 
HT_{\overline{1_{h_0}}}(\pi_v,,\Pi_{t_0} \otimes \st_{t-t_0}(\pi_v \{ -t_0\frac{g-1}{2} \})) \bigr ) \simeq \\
\lexp p j^{=tg}_{\overline{1_{h_0}} \setminus \overline{1_{h_0+1}},!*} 
HT_{\overline{1_{h_0}} \setminus \overline{1_{h_0+1}}}(\pi_v,,\Pi_{t_0} \otimes 
\st_{t-t_0}(\pi_v \{ -t_0\frac{g-1}{2} \})),
\end{multline*}
il s'agit simplement d'identifier 
\addtocounter{smfthm}{1}
\begin{equation} \label{eq-terme2}
i^{tg+1}_{\overline{1_{h_0+1}},*} \lexp p \hi^{-1} i^{tg+1,*}_{\overline{1_{h_0+1}}} 
\bigl ( \lexp p j^{=tg}_{\overline{1_{h_0}}!*} HT_{\overline{1_{h_0}}}(\pi_v,,\Pi_{t_0} \otimes 
\st_{t-t_0}(\pi_v \{ -t_0\frac{g-1}{2} \})) \bigr )
\end{equation} 
avec le premier terme de la suite exacte courte de l'ÈnoncÈ. 
La preuve se dÈroule en deux temps:
\begin{itemize}
\item[(i)] on construit tout d'abord une surjection du faisceau pervers $P$ de (\ref{eq-terme2}) 
vers le faisceau pervers $A$ de l'ÈnoncÈ

\item[(ii)] et on vÈrifie dans un deuxiËme temps que les faisceaux de cohomologie de
ces deux faisceaux pervers possËdent les mÍmes germes en tout point gÈomÈtrique.
\end{itemize}
Notons alors $K$ le noyau de cette surjection $P \twoheadrightarrow A$ et soit pour 
$i_z:\overline \{z \} \hookrightarrow X_{\IC,\bar s_v}$ un point
gÈomÈtrique, le foncteur $\hi^*_z$ qui ‡ un faisceau pervers $C$ associe, dans le groupe de Grothendieck,
la somme alternÈe $\sum_{i \in \Zm} (-1)^i \hi^i i_z^* C$ des germes de ses faisceaux de cohomologie.
On a ainsi $\hi^*_z(K)=\hi^*_z(P)-\hi^*_z(A)$ qui d'aprËs (ii) est nul, ce qui
implique que $K$ est nul puisque l'image d'un faisceau pervers $C$ dans le groupe de Grothendieck est 
dÈterminÈe par ses $\hi^*_z C$  lorsque $z$ dÈcrit tous les points gÈomÈtriques, 
cf. par exemple la section 7 de \cite{boyer-invent2}.

\medskip

(i) Pour le premier point notons tout d'abord que d'aprËs (\ref{eq00}), on a\footnote{cf. la remarque
prÈcÈdente pour calculer les dÈcalages}
\begin{multline*}
j^{=tg}_{\overline{1_{h_0}},!} HT_{\overline{1_{h_0}}}(\pi_v,\Pi_{t_0} \otimes 
\st_{t-t_0}(\pi_v \{ -t_0\frac{g-1}{2} \}))=
\sum_{i=t}^s  \lexp p j^{=ig}_{\overline{1_{h_0}},!*} HT_{\overline{1_{h_0}}} \Bigl ( \pi_v, \\
\Pi_{t_0} \{ (i-t) \frac{g-1}{2} \} \otimes \bigl ( \st_{t-t_0} (\pi_v\{ -\frac{i-t}{2} \}) \times 
\st_{i-t}(\pi_v \{ \frac{t-t_0}{2} \} ) \bigr )  \{ -t_0\frac{g-1}{2} \} \Bigr ) (\frac{i-t}{2}).$$
\end{multline*}
En utilisant la filtration par les poids et la surjection 
$$\st_{t-t_0} (\pi_v\{-\frac{i-t}{2} \}) \times \st_{i-t}(\pi_v\{ \frac{t-t_0}{2} \}) \twoheadrightarrow \st_{i-t_0}(\pi_v),$$
on obtient une surjection
\begin{multline} \label{eq-noyau}
\lexp p \hi^{-1} i^{tg+1,*}_{\overline{1_{h_0}}} \bigl (  \lexp p j^{=tg}_{\overline{1_{h_0}},!*} 
HT_{\overline{1_{h_0}}}(\pi_v,,\Pi_{t_0} \otimes \st_{t-t_0}(\pi_v \{ -t_0\frac{g-1}{2} \}) ) \bigr ) 
\twoheadrightarrow \\
\lexp p j^{=(t+1)g}_{\overline{1_{h_0}},!*} HT_{\overline{1_{h_0}}}(\pi_v, ,\Pi_{t_0} \{ \frac{g-1}{2} \}
\otimes \st_{t+1-t_0}(\pi_v \{ -t_0\frac{g-1}{2} \})) (\frac{1}{2})
\end{multline}
ce qui, aprËs par application du foncteur
$\lexp p i^{h+1 \leq +0,*}_{\overline{1_{h_0}},\overline{1_{h_0+1}}}$, fournit la surjection cherchÈe.
En effet pour le membre de droite, on a
$\lexp p i^{h+1 \leq +0,*}_{\overline{1_{h_0}},\overline{1_{h_0+1}}} \circ
\lexp p j^{=(t+1)g}_{\overline{1_{h_0}},!*}  \simeq \lexp p j^{=(t+1)g}_{\overline{1_{h_0+1}},!*}$,
alors que pour le membre de gauche, cf. la preuve du lemme \ref{lem-HT1}, on a
$$\lexp p i^{tg+1 \leq +0,*}_{\overline{1_{h_0}},\overline{1_{h_0+1}}}  \circ
\lexp p \hi^{-1} i^{tg+1,*}_{\overline{1_{h_0}}} \circ  \lexp p j^{=tg}_{\overline{1_{h_0}},!*}  \simeq
\lexp p \hi^{-1} i^{tg+1,*}_{\overline{1_{h_0+1}}} \circ \lexp p j^{=tg}_{\overline{1_{h_0}}!*}.$$

\noindent (ii) Les germes des faisceaux de cohomologie de
$$A:=\lexp p j^{=(t+1)g}_{\overline{1_{h_0+1}},!*} HT_{\overline{1_{h_0+1}}}
\bigr (\pi_v, \Pi_{t_0} \{ \frac{g-1}{2} \} \otimes \st_{t+1-t_0}
(\pi_v \{ -t_0\frac{g-1}{2} \})_{P_{1,(t+1-t_0)g}(F_v)} \bigl ) (\frac{1}{2})$$
sont donnÈs par le corollaire prÈcÈdent de mÍme que ceux de 
$$B:=\lexp p j^{=tg}_{\overline{1_{h_0}} \setminus \overline{1_{h_0+1}},!*} 
HT_{\overline{1_{h_0}} \setminus \overline{1_{h_0+1}}} \bigl ( \pi_v,\Pi_{t_0} \otimes 
\st_{t-t_0}(\pi_v \{ -t_0\frac{g-1}{2} \}) \bigr ).$$
Il s'agit alors de vÈrifier, qu'‡ dÈcalage d'un indice prËs, pour tout point gÈomÈtrique $z$ de 
$X^{\geq h+1}_{\IC,\bar s_v,\overline{1_{h_0+1}}}$, on obtient la mÍme chose. Pour ce qui concerne
les conditions d'annulation, on vÈrifie aisÈment qu'elles sont les mÍmes. 
Soit alors $z$ un point gÈomÈtrique de $X^{=(t+\delta)g}_{\IC,\bar s_v,\overline{1_{(t+\delta)g}}}$,
d'aprËs le corollaire prÈcÈdent.
\begin{itemize}
\item  La fibre en $z$ du faisceau de cohomologie d'indice 
$(t+\delta)g-d-\delta$ de $B$ se calcule par induction ‡ partir de celle de
$\lexp p j^{=tg}_{a,!*} HT_a \bigl ( \pi_v,\Pi_{t_0} \otimes \st_{t-t_0}(\pi_v \{ -t_0\frac{g-1}{2} \}) \bigr )$ 
o˘ $a$ est le sous-espace vectoriel de $F_v^d$ engendrÈ par les vecteurs 
$e_1,\cdots,e_{h_0},e_{h_0+2},\cdots,e_{h+1}$ de la base canonique. D'aprËs le corollaire
prÈcÈdent, la fibre en $z$ de $\hi^{(t+\delta)g-d-\delta} \lexp p j^{=tg}_{a,!*} 
HT_a \bigl ( \pi_v,\Pi_{t_0} \otimes \st_{t-t_0}(\pi_v \{ -t_0\frac{g-1}{2} \}) \bigr )$ est isomorphe ‡
$$\Pi_{t_0} \{ \delta \frac{g-1}{2} \} \otimes \Bigl ( \st_{t-t_0}(\pi_v \{ \delta \frac{g-1}{2} \} \otimes 
\bigl ( \speh_\delta(\pi_v \{ -(t-t_0) \frac{g-1}{2} \}) \bigr )_{|P_{1,\delta g}(F_v)} \Bigr )  
\{ -t_0\frac{g-1}{2} \}$$
relativement au Levi
$$\left ( \begin{array}{cccc} GL_{t_0g} & 0 & 0 & 0 \\ 0 & 1 & 0 & 0 \\ 0 & 0 & GL_{(t-t_0)g} & 0 \\
0 & 0 & 0 & GL_{\delta g-1} \end{array} \right )$$
de $GL_{(t+\delta)g}$.
En induisant ‡ $P_{tg_0,t_0g+1,(t+\delta)g}(F_v)$, on obtient que la fibre en $z$ de 
$\hi^{(t+\delta)g-d-\delta} B$ est isomorphe ‡
$$\Pi_{t_0} \{ \delta \frac{g-1}{2} \} \otimes \Bigl ( \st_{t-t_0}(\pi_v \{ -\frac{\delta}{2} \} ) \times \bigl (
\speh_{\delta}(\pi_v \{ \frac{t-t_0}{2} \} ) \bigr )_{|P_{1,\delta g}(F_v)} \Bigr ) \{ -t_0 \frac{g-1}{2} \}$$
et donc d'aprËs le deuxiËme isomorphisme du lemme \ref{lem-mirabolique}, ‡
$$\Pi_{t_0} \{ \delta \frac{g-1}{2} \} \otimes LT_{\pi_v}(t-t_0,\delta-1)_{|P_{1,(t-t_0+\delta)g}(F_v)} 
\{ -t_0 \frac{g-1}{2} \}.$$

\item La fibre en $z$ de $\hi^{(t+1+\delta-1)g-d-\delta+1} A$ est d'aprËs le corollaire prÈcÈdent
isomorphe ‡
$$\Pi_{t_0} \{ \delta \frac{g-1}{2} \} \otimes \Bigl ( \bigl ( \st_{t-t_0+1}(\pi_v \{ -\frac{\delta-1}{2} \}
) \bigr )_{P_{1,(t+1-t_0)g}(F_v)} \times \speh_{\delta-1}(\pi_v \{ \frac{t-t_0+1}{2} \}) \Bigr ) 
\{ -t_0 \frac{g-1}{2} \}$$
lequel d'aprËs le premier isomorphisme du lemme \ref{lem-mirabolique} est isomorphe ‡
$$\Pi_{t_0} \{ \delta \frac{g-1}{2} \} \otimes  LT_{\pi_v}(t-t_0,\delta-1)_{|P_{1,(t-t_0+\delta)g}(F_v)} 
\{ -t_0 \frac{g-1}{2} \}.$$
\end{itemize}
Ainsi puisque la fibre en $z$ de $\hi^i B$ est isomorphe ‡ celle de
$$\hi^{i+1} \Bigl ( \lexp p \hi^{-1} i^{tg+1,*}_{\overline{1_{h_0}}} \bigl (  \lexp p j^{=tg}_{\overline{1_{h_0}},!*} 
HT_{\overline{1_{h_0}}}(\pi_v,,\Pi_{t_0} \otimes \st_{t-t_0}(\pi_v \{ -t_0\frac{g-1}{2} \}) ) \bigr )  \Bigr ),$$
on en dÈduit donc que le noyau de la surjection (\ref{eq-noyau}) est un faisceau pervers dont les 
fibres en $z$
de ses faisceaux de cohomologie sont nulles. Par symÈtrie, le rÈsultat est valable pour tout
point gÈomÈtrique de $X^{\geq h+1}_{\IC,\bar s_v,\overline{1_{h_0+1}}}$ de sorte que ce noyau est nul,
et la surjection (\ref{eq-noyau}) est en fait un isomorphisme, d'o˘ le rÈsultat.
\end{proof}

\begin{proof} \textit{de la proposition \ref{prop-fond1}.}
On raisonne par rÈcurrence sur $t$ de $s$ ‡ $1$. Le cas $t=s$ Ètant trivial puisque 
$P(\pi_v,\Pi_s)$ est nul, on suppose le rÈsultat acquis jusqu'au rang $t+1$.
Partons de l'ÈgalitÈ de \cite{boyer-invent2} 4.5.1 rappelÈe plus haut en (\ref{eq00}):
\addtocounter{smfthm}{1}
\begin{equation} \label{eq-P}
P(\pi_v,\Pi_t) =  \sum_{i=t+1}^{s} \lexp p j^{=ig}_{\overline{1_h},!*} HT_{\overline{1_h}} \Bigl (
\pi_v,\Pi_t \{ (i-t) \frac{g-1}{2} \} \otimes \st_{i-t}(\pi_v \{ -t \frac{g-1}{2} \} ) \Bigr ) (\frac{i-t}{2}).
\end{equation}
On peut alors appliquer, en utilisant l'hypothËse de rÈcurrence pour $i \geq t+1$, 
le corollaire prÈcÈdent et calculer, en utilisant l'exactitude du foncteur
$j^{=tg}_{\overline{1_{h}} \setminus \overline{1_{h+1}},!} 
j^{=tg,*}_{\overline{1_{h}} \setminus \overline{1_{h+1}}}$,
l'image, dans le groupe de Grothendieck des faisceaux pervers munis d'une action par correspondances
de $G(\Am^{\oo,v}) \times P_{h,h+1,d}(F_v) \times \Zm$, de $P(\pi_v,\Pi_t)$ par ce foncteur, soit en notant
que pour tout $i \geq t$, on a trivialement
$( j^{=tg}_{\overline{1_{h}} \setminus \overline{1_{h+1}},!} 
j^{=tg,*}_{\overline{1_{h}} \setminus \overline{1_{h+1}}} ) \circ \lexp p j^{=ig}_{\overline{1_h},!*}
= (j^{=ig}_{\overline{1_{h}} \setminus \overline{1_{h+1}},!} 
j^{=ig,*}_{\overline{1_{h}} \setminus \overline{1_{h+1}}}) \circ \lexp p j^{=ig}_{\overline{1_h},!*}$
\begin{multline*}
j^{=tg}_{\overline{1_{h}} \setminus \overline{1_{h+1}},!}
j^{=tg,*}_{\overline{1_{h}} \setminus \overline{1_{h+1}}} P(\pi_v,\Pi_t) =  \sum_{i=t+1}^{s-1} \\
 \lexp p j^{=ig}_{\overline{1_{h}} \setminus \overline{1_{h+1}},!*}
HT_{\overline{1_{h}} \setminus \overline{1_{h+1}}} (\pi_v,\Pi_{t} \{ (i-t) \frac{g-1}{2} \} \otimes
\st_{i-t}(\pi_v \{ -t \frac{g-1}{2} \})_{|P_{1,(i-t)g}(F_v)}) (\frac{i-t}{2}) + \\
\lexp p j^{=(i+1)g}_{\overline{1_{h+1}},!*}
HT_{\overline{1_{h+1}}} (\pi_v,\Pi_t \{ (i+1-t)\frac{g-1}{2} \} \otimes
\st_{i+1-t}(\pi_v \{  -t \frac{g-1}{2} \} )_{|P_{1,(i-t+1)g}(F_v)}) (\frac{i+1-t}{2}).
\end{multline*}
En regroupant les termes dans l'Èquation prÈcÈdente, on obtient alors
\begin{multline*}
j^{=tg}_{\overline{1_{h}} \setminus \overline{1_{h+1}},!}
j^{=tg,*}_{\overline{1_{h}} \setminus \overline{1_{h+1}}} P(\pi_v,\Pi_t) = 
 \lexp p j^{=(t+1)g}_{\overline{1_{h}} \setminus \overline{1_{h+1}},!*}
HT_{\overline{1_{h}} \setminus \overline{1_{h+1}}} (\pi_v,\Pi_{t} \{ \frac{g-1}{2} \} \otimes
(\pi_v \{ -t \frac{g-1}{2} \})_{|P_{1,g}(F_v)}) (\frac{1}{2}) + \\
 \sum_{i=t+2}^{s}  \lexp p j^{=ig}_{\overline{1_h},!*} HT_{\overline{1_h}} \Bigl (
\pi_v,\Pi_t \{ (i-t) \frac{g-1}{2} \} \otimes \st_{i-t}(\pi_v \{ -t \frac{g-1}{2} \} )_{|P_{1,(i-t)g}(F_v)} \Bigr ) 
(\frac{i-t}{2}).
\end{multline*}
Ainsi en soustrayant ‡ l'ÈgalitÈ (\ref{eq-P}) l'ÈgalitÈ ci-avant, on obtient d'aprËs (\ref{eq-sec1}), 
que l'image de $i^{h+1}_{\overline{1_{h+1}},*} \lexp p i^{h+1,*}_{\overline{1_{h+1}}} P(\pi_v,\Pi_t)$
dans le groupe de Grothendieck des faisceaux pervers munis d'une action par correspondances
de $G(\Am^{\oo,v}) \times P_{h,h+1,d}(F_v) \times \Zm$, est Ègale ‡ 
\begin{multline*}
\lexp p j^{=(t+1)g}_{\overline{1_h},!*} HT_{\overline{1_h}} \Bigl (
\pi_v,\Pi_t \{ \frac{g-1}{2} \} \otimes (\pi_v \{ -t \frac{g-1}{2} \})_{|P_{1,g}(F_v)}  \Bigr ) (\frac{1}{2})  \\ -
 \lexp p j^{=(t+1)g}_{\overline{1_{h}} \setminus \overline{1_{h+1}},!*}
HT_{\overline{1_{h}} \setminus \overline{1_{h+1}}} (\pi_v,\Pi_{t} \{ \frac{g-1}{2} \} \otimes
(\pi_v \{ -t \frac{g-1}{2} \}))_{|P_{1,g}(F_v)} (\frac{1}{2}) \\
= \lexp p j^{=(t+1)g}_{\overline{1_{h+1}},!*} HT_{\overline{1_h}} \Bigl (
\pi_v,\Pi_t \{ \frac{g-1}{2} \} \otimes (\pi_v \{ -t \frac{g-1}{2} \})_{|P_{1,g}(F_v)}  \Bigr ) (\frac{1}{2}).
\end{multline*}
On a alors le diagramme suivant
$$\xymatrix{
& P(\pi_v,\Pi_t) \ar@{->>}[r] & i^{h+1}_{\overline{1_{h+1}},*} \lexp p i^{h+1,*}_{\overline{1_{h+1}}} P(\pi_v,\Pi_t) 
\\
K \ar@{^{(}->}[r] \ar@{-->}[urr]^0 & P(\pi_v,\Pi_t) \ar@{=}[u] \ar@{->>}[r] & 
\lexp p j^{=(t+1)g}_{\overline{1_{tg+1}},!*} 
HT_{\overline{1_{tg+1}}} (\pi_v,\Pi_t\{ \frac{g-1}{2} \} \otimes (\pi_v)_{|P_{1,g}(F_v)} \{ -t \frac{g-1}{2} \} ) 
(\frac{1}{2}) \ar@{-->}[u]^\simeq
}$$
o˘ la flËche composÈe $K \longrightarrow i^{h+1}_{\overline{1_{h+1}},*} \lexp p i^{h+1,*}_{\overline{1_{h+1}}} 
P(\pi_v,\Pi_t)$ est nÈcessairement nulle. En effet
\begin{itemize}
\item d'un cÙtÈ l'image d'un faisceau pervers dans le groupe de Grothendieck des faisceaux pervers,
est uniquement dÈterminÈe par les fibres de la somme alternÈe de ses faisceaux de cohomologie, de
sorte que les constituants simples de $i^{h+1}_{\overline{1_{h+1}},*} \lexp p i^{h+1,*}_{\overline{1_{h+1}}} 
P(\pi_v,\Pi_t)$ sont des extensions intermÈdiaires de systËmes locaux sur $X^{=(t+1)g}_{\IC,\bar s_v,
\overline{1_{tg+1}}}$,

\item de l'autre les facteurs simples de $K$
sont ‡ support soit dans $X^{\geq h+2}_{\IC,\bar s_v}$ soit des extensions intermÈdaires
de systËmes locaux sur $X^{= (t+1)g}_{\IC,\bar s_v, \overline{1_{tg}} \setminus \overline{1_{tg+1}}}$.
\end{itemize}
Enfin on rÈcupËre l'action complËte de $P_{tg,d}(F_v)$ en utilisant comme 
prÈcÈdemment la remarque de la preuve du corollaire \ref{coro-hi}.
\end{proof}

\section{Faisceaux de cohomologie du complexe des cycles Èvanescents}

En ce qui concerne les fibres des faisceaux de cohomologie de $\Psi_{\IC}$, on peut
mener le mÍme raisonnement. Rappelons qu'en notant
$$\bar j: X_{\IC,\bar \eta_v} \hookrightarrow X_{\IC} \hookleftarrow X_{\IC,\bar s_v}: i,$$
o˘ $\bar j_!=\lexp p {\bar j_!}$ et  $\bar j_*=\lexp p {\bar j_*}$, le complexe des
cycles Èvanescents est $\lexp p \hi^{-1} i^* \bar j_* \bar \Qm_l$. Ainsi en utilisant que l'inclusion
$$\bar j^{\geq 1}_{\neq \overline 1_1}:X_{\IC,\bar \eta}-X^{\geq 1}_{\IC,\bar s_v,\overline 1_1} \hookrightarrow
X_{\IC,\bar \eta}$$
est affine, exactement les mÍmes arguments que ceux du lemme \ref{lem-HT1} montrent que
$i^{1,*}_{\overline 1_1} \Psi_{\IC}$ 
est pervers, ce qui donne la suite exacte courte
\addtocounter{smfthm}{1}
\begin{equation} \label{eq-sec-psi}
0 \rightarrow j^{\geq 1}_{\neq \overline 1_1,!} j^{\geq 1,*}_{\neq \overline 1_1} \Psi_{\IC} 
\longrightarrow \Psi_{\IC} \longrightarrow i^1_{\overline 1_1,*} 
\lexp p \hi^0 i^{1,*}_{\overline 1_1} \Psi_{\IC} \rightarrow 0.
\end{equation}
D'aprËs \cite{boyer-invent2}, l'image de $\Psi_{\IC}$ dans
le groupe de Grothendieck est dÈterminÈe via la formule des traces.
Afin d'exprimer le rÈsultat, il est tout d'abord plus simple de dÈcomposer
$\Psi_{\IC}$ comme dans \cite{boyer-invent2} sous la forme
$$\Psi_{\IC} \simeq \bigoplus_{\pi_v} \Psi_{\IC,\pi_v}$$
o˘ $\pi_v$ dÈcrit les classes d'Èquivalence inertielle des reprÈsentations irrÈductibles cuspidales
de $GL_g(F_v)$ pour $1 \leq g \leq d$. En posant 
$s=\lfloor \frac{d}{g} \rfloor$, dans le groupe de Grothendieck adÈquat, on a alors, 
cf. \cite{boyer-invent2} proposition 4.1.4 ou corollaire 5.4.2
\addtocounter{smfthm}{1}
\begin{equation} \label{eq-psi-groth}
\Psi_{\IC,\pi_v}= \sum_{k=1-s}^{s-1} \sum_{\atop{|k| < t \leq s}{t \equiv k-1 \mod 2}}
\PC(t,\pi_v)(-\frac{k}{2}).
\end{equation}

\begin{prop} \label{prop-psi-fil} (cf. aussi la proposition B.3.4 de \cite{boyer-duke}) \\
Avec les notations prÈcÈdentes, on a
$$i^1_{\overline 1_1,*} \lexp p \hi^0 i^{1,*}_{\overline 1_1} \Psi_{\IC,\pi_v} =
\sum_{k=1}^{s} 
\lexp p j^{=kg}_{\overline{1_1},!*} HT_{\overline{1_1}} (\pi_v,\st_k(\pi_v)_{|P_{1,kg}(F_v)}) 
(\frac{1-k}{2}).$$
\end{prop}

\begin{proof}
Les arguments sont similaires ‡ ceux de la proposition \ref{prop-fond1} en utilisant le corollaire
\ref{coro-j} avec $h_0+1=1$ en convenant que 
$$X^{=0}_{\IC,\bar s_v,\overline{1_0}}:=X_{\IC,\bar \eta_v}.$$ 
On applique ainsi le foncteur 
$j^{\geq 1}_{\neq \overline 1_1,!} j^{\geq 1,*}_{\neq \overline 1_1}$ ‡ l'ÈgalitÈ (\ref{eq-psi-groth}), ce
qui donne d'aprËs le corollaire \ref{coro-j}
\begin{multline*}
j^{\geq 1}_{\neq \overline 1_1,!} j^{\geq 1,*}_{\neq \overline 1_1} \Psi_{\IC,\pi_v} =
\sum_{k=1-s}^{s-1} \sum_{\atop{|k| < t < s}{t \equiv k-1 \mod 2}} \Bigl (
\lexp p j^{=tg}_{\neq \overline{1_1},!*} HT_{\neq \overline{1_1}}(\pi_v,\st_t(\pi_v)_{|P_{1,tg}(F_v)}) 
\otimes \Lm(\pi_v) (-\frac{k}{2}) \\
+ \lexp p j^{=(t+1)g}_{\overline{1_1},!*} HT_{\overline{1_1}}(\pi_v,\st_{t+1}(\pi_v)_{|P_{1,(t+1)g}(F_v)})
\otimes \Lm(\pi_v) (- \frac{k-1}{2}) \Bigr ) \\
+ \sum_{\atop{k \equiv s-1 \mod 2}{|k| \leq s-1}} \PC(s,\pi_v) (-\frac{k}{2})
\end{multline*}
On soustrait alors cette ÈgalitÈ ‡ (\ref{eq-psi-groth}), ce qui d'aprËs la suite exacte courte (\ref{eq-sec-psi})
fournit le rÈsultat.
\end{proof}

\begin{coro} \label{coro-fin}
Soient $\pi_v$ une reprÈsentation irrÈductible cuspidale de $GL_g(F_v)$ et
$z$ un point gÈomÈtrique de $X^{=d}_{\IC,\bar s_v}$. La fibre en $z$ de
$\hi^i \Psi_{\IC,\pi_v}$ vÈrifie alors les points suivant:
\begin{itemize}
\item elle est nulle si $g$ ne divise pas $d$;

\item si $d=sg$, elle est nulle si $i>0$ ou si $i \leq -s$;

\item pour $d=sg$ et $1-s \leq i \leq 0$, elle est isomorphe ‡ la fibre en $z$ de
$HT(\pi_v,LT_{\pi_v}(s-1+i,-i)) \otimes \Lm(\pi_v) (\frac{i}{2})$.
\end{itemize}
\end{coro}

\rem Le corollaire ci-dessus correspond au rÈsultat principal de \cite{boyer-invent2}, cf. le corollaire
2.2.10 de loc. cit. En utilisant le thÈorËme de comparaison de Berkovich, on en dÈduit
le calcul des groupes de cohomologie des espaces de Lubin-Tate, cf. le 
thÈorËme 2.3.5 de \cite{boyer-invent2}.

\begin{proof}
En utilisant la suite exacte courte (\ref{eq-sec-psi}), le germe cherchÈ est celui de 
$i^1_{\overline 1_1,*} \lexp p \hi^0 i^{1,*}_{\overline 1_1} \Psi_{\IC,\pi_v}$.
ConsidÈrons alors la filtration par les poids de ce dernier faisceau pervers
$$0=\Fil^0( \Psi_{\IC,\pi_v}) \subset \Fil^1( \Psi_{\IC,\pi_v}) \subset \cdots \subset \Fil^s( \Psi_{\IC,\pi_v}),$$
dont les graduÈs sont 
$$\gr^k( \Psi_{\IC,\pi_v}):=\Fil^k(\Psi_{\IC,\pi_v})/\Fil^{k-1}(\Psi_{\IC,\pi_v}) \simeq
\lexp p j^{=kg}_{\overline{1_1},!*} HT_{\overline{1_1}} (\pi_v,\st_k(\pi_v)_{|P_{1,kg}(F_v)}) (\frac{1-k}{2}),$$
pour $k=1,\cdots,s$. On considËre alors la suite spectrale calculant les faisceaux de cohomologie
de $\Psi_{\IC,\pi_v}$ ‡ partir de ceux de ses graduÈs, i.e. 
$$E_1^{i,j}=\hi^{i+j} \Bigl (
\lexp p j^{=-ig}_{\overline{1_1},!*} HT_{\overline{1_1}} (\pi_v,\st_{-i}(\pi_v)_{|P_{1,-ig}(F_v)}) 
(\frac{1+i}{2}) \Bigr ) \Rightarrow \hi^{i+j} \Psi_{\IC,\pi_v}.$$
D'aprËs le corollaire \ref{coro-hi}, la fibre en $z$ des $E_1^{i,j}$ vÈrifie alors les propriÈtÈs suivantes:
\begin{itemize}
\item tous les $(E_1^{i,j})_z$ sont nuls si  $i \geq 0$ ou si $i< -s$;

\item tous les $(E_1^{i,j})_z$ sont nuls si $g$ ne divise pas $d$;

\item pour $d=sg$, les $(E_1^{i,j})_z$ sont nuls si $i+j >0$ ou si $i+j \leq -s$.

\item Pour $d=sg$ et $1-s \leq i+j \leq 0$, les $(E_1^{i,j})_z$ sont nuls sauf pour $-i=s+i+j$, i.e. $j=-s-2i$
auquel cas cette fibre est isomorphe ‡ la fibre en $z$ de $HT(\pi_v,LT_{\pi_v}(s-1+i+j,-i_j)) \otimes \Lm(\pi_v) (\frac{i+j}{2})$.
\end{itemize}
Ainsi les couples $(i,j)$ tels que $(E_1^{i,j})_z$ est non nul sont les $(-s+\delta,s-2 \delta)$ pour 
$\delta=0,1,\cdots,s-1$. On remarque alors que toutes les flËches $d_r^{i,j}$ de cette suite spectrale 
sont nÈcessairement nuls puisque, par une rÈcurrence immÈdiate sur $r$, soit son espace de dÈpart $E_r^{i,j}$ 
soit son espace d'arrivÈe $E_r^{i+r,j+r-1}$ est nul. Ainsi donc on a $(E_\oo^{i,j})_z \simeq (E_1^{i,j})_z$
et le rÈsultat dÈcoule des tirets prÈcÈdents.
\end{proof}

\bibliographystyle{plain}
\bibliography{bib-ok}

\end{document}